\documentclass[11pt]{article}
\usepackage[utf8]{inputenc}
\usepackage[paperwidth=8.5in,
            paperheight=11in,
            portrait,
            top=1in,
            bottom=1.in,
            left=1.15in,
            right=1.15in]{geometry}
\usepackage{authblk}

\usepackage{amsfonts,amsmath,amssymb,amsthm}
\usepackage{latexsym,mathrsfs,mathtools,bm}
\usepackage{graphicx,subcaption,epsfig,caption,float,xcolor}
\usepackage{enumitem}
  
\usepackage{tikz}
\usepackage[all]{xy}
\usetikzlibrary{calc}
\usepackage{chngcntr}

\usepackage[hidelinks]{hyperref}
\usepackage{bookmark}

\theoremstyle{plain}
\newtheorem{thm}{Theorem}[section]
\newtheorem{conj}[thm]{Conjecture}
\newtheorem{lem}[thm]{Lemma}
\newtheorem{prop}[thm]{Proposition}
\newtheorem{coro}[thm]{Corollary}

\newtheorem{defi}[thm]{Definition}
\newtheorem{rem}[thm]{Remark}

\numberwithin{equation}{section}

\newcommand{\al}{\alpha}
\newcommand{\si}{\sigma}
\newcommand{\om}{\omega}

\newcommand{\cA}{\mathcal{A}}

\newcommand{\cC}{\mathcal{C}}

\newcommand{\bC}{\mathbb{C}}
\newcommand{\fS}{\mathfrak{S}}

\newcommand{\cH}[1]{H_{\alpha_1,\alpha_2,#1}}
\newcommand{\fH}{H_{k,n}}

\newcommand{\cosets}{\fS_k\backslash\fS_{k+n}/\fS_k}

\newcommand{\avB}{B_n(\bar{1}\bar{2})}

\newcommand{\rmspan}{\mathrm{span}}

\newcommand{\diag}[3]{ \foreach \t in {1,...,#3} {\draw[thick] (#1+\t,#2-1) rectangle (#1+\t-1,#2);} }

\newcommand{\mdiagp}[3]{ 
\foreach \t in {1,...,#3} {\filldraw[thick, fill=lightgray, draw=black] (#1+\t,#2-1) rectangle (#1+\t-1,#2);}}

\newcommand{\mdiag}[3]{ 
\foreach \t in {1,...,3} {\filldraw[thick, fill=lightgray, draw=black] (#1+\t,#2-1) rectangle (#1+\t-1,#2);}
\foreach \t in {4,...,#3} {\draw[thick] (#1+\t,#2-1) rectangle (#1+\t-1,#2);}}

\newcommand{\diagg}[4]{ \foreach \t in {1,...,#3} {\draw[thick] (#1+\t,#2-1) rectangle (#1+\t-1,#2);} \foreach \t in {1,...,#4} {\draw[thick] (#1+\t,#2-1) rectangle (#1+\t-1,#2-2);} }

\newcommand{\mdiaggp}[4]{ \foreach \t in {1,...,#3} {\filldraw[thick, fill=lightgray, draw=black] (#1+\t,#2-1) rectangle (#1+\t-1,#2);}
\foreach \t in {1,...,#4} {\draw[thick] (#1+\t,#2-1) rectangle (#1+\t-1,#2-2);} }

\newcommand{\mdiagg}[4]{ \foreach \t in {1,...,3} {\filldraw[thick, fill=lightgray, draw=black] (#1+\t,#2-1) rectangle (#1+\t-1,#2);}
\foreach \t in {4,...,#3} {\draw[thick] (#1+\t,#2-1) rectangle (#1+\t-1,#2);}
\foreach \t in {1,...,#4} {\draw[thick] (#1+\t,#2-1) rectangle (#1+\t-1,#2-2);} }

\newcommand{\diaggg}[5]{ \foreach \t in {1,...,#3} {\draw[thick] (#1+\t,#2-1) rectangle (#1+\t-1,#2);} \foreach \t in {1,...,#4} {\draw[thick] (#1+\t,#2-1) rectangle (#1+\t-1,#2-2);}
                         \foreach \t in {1,...,#5} {\draw[thick] (#1+\t,#2-2) rectangle (#1+\t-1,#2-3);} }
                     
\newcommand{\mdiagggp}[5]{ \foreach \t in {1,...,#3} {\filldraw[thick, fill=lightgray, draw=black] (#1+\t,#2-1) rectangle (#1+\t-1,#2);}
\foreach \t in {1,...,#4} {\draw[thick] (#1+\t,#2-1) rectangle (#1+\t-1,#2-2);}
                         \foreach \t in {1,...,#5} {\draw[thick] (#1+\t,#2-2) rectangle (#1+\t-1,#2-3);} }                        
                         
\newcommand{\mdiaggg}[5]{ \foreach \t in {1,...,3} {\filldraw[thick, fill=lightgray, draw=black] (#1+\t,#2-1) rectangle (#1+\t-1,#2);}
\foreach \t in {4,...,#3} {\draw[thick] (#1+\t,#2-1) rectangle (#1+\t-1,#2);}
\foreach \t in {1,...,#4} {\draw[thick] (#1+\t,#2-1) rectangle (#1+\t-1,#2-2);}
                         \foreach \t in {1,...,#5} {\draw[thick] (#1+\t,#2-2) rectangle (#1+\t-1,#2-3);} }

\newcommand{\mdiaggggp}[6]{ \foreach \t in {1,...,#3} {\filldraw[thick, fill=lightgray, draw=black] (#1+\t,#2-1) rectangle (#1+\t-1,#2);}
\foreach \t in {1,...,#4} {\draw[thick] (#1+\t,#2-1) rectangle (#1+\t-1,#2-2);}
                         \foreach \t in {1,...,#5} {\draw[thick] (#1+\t,#2-2) rectangle (#1+\t-1,#2-3);} \foreach \t in {1,...,#6} {\draw[thick] (#1+\t,#2-3) rectangle (#1+\t-1,#2-4);} }

\def\eseq{-10} 

\newcommand{\strand}[2]{
	\fill (#1,#2) circle (0.2);
	\draw[thick] (#1,#2) -- ++(0,-4);
	\fill (#1,#2-4) circle (0.2);
} 

\newcommand{\slab}[2]{
	\node at #1 {\scriptsize #2};
} 

\newcommand{\scirc}[2]{
	\fill (#1,#2) circle (0.2);
} 

\newcommand{\ocross}[2]{
\scirc{#1}{#2}
\scirc{#1}{#2-4}
\scirc{#1+4}{#2}
\scirc{#1+4}{#2-4}
\draw[thick] (#1+4,#2)..controls +(0,-2) and +(0,+2) .. (#1,#2-4);
\fill[white] (#1+2,#2-2) circle (0.4);
\draw[thick] (#1,#2)..controls +(0,-2) and +(0,+2) .. (#1+4,#2-4);
} 

\newcommand{\ellk}[2]{
	\draw[white,fill=lightgray] (#1-1,#2-2) rectangle (#1+1,#2+2);
	\fill (#1,#2+2) ellipse (1.4cm and 0.2cm);
	\fill (#1,#2-2) ellipse (1.4cm and 0.2cm);
	\draw[thick] (#1-1,#2+2) -- (#1-1,#2-2);
	\slab{(#1,#2)}{$k$}
} 

\newcommand{\ellstrand}[2]{
	\draw[thick] (#1+1,#2+2) -- (#1+1,#2-2);
} 

\newcommand{\ellocross}[2]{
	\draw[thick] (#1+4,#2+2)..controls +(0,-2) and +(0,+2) .. (#1+1,#2-2);
	\fill[white] (#1+2.5,#2) circle (0.3);
	\draw[thick] (#1+1,#2+2)..controls +(0,-2) and +(0,+2) .. (#1+4,#2-2);
	\fill (#1+4,#2+2) circle (0.2);
	\fill (#1+4,#2-2) circle (0.2);
} 

\newcommand{\encell}[2]{
	\fill (#1,#2) circle (0.2);	
			\draw[white,fill=white] (#1-4-1.1,#2-2-1) rectangle (#1-4+1.1,#2-2-0.6);		
			\draw[thick] (#1,#2) arc (0:-95:1.2) -- (#1-4+1.4,#2-2+0.8) (#1-4-1.4,#2-2+0.8) arc (90:270:0.8) -- (#1-4+2.8,#2-2-0.8) arc (90:0:1.2);
	\fill (#1,#2-4) circle (0.2);
} 

\newcommand{\encells}[2]{
	\fill (#1,#2) circle (0.2);	
	\draw[white,fill=white] (#1-2*4-1.1,#2-2-1) rectangle (#1+1.1,#2-2-0.6);		
	\draw[thick] (#1,#2) arc (0:-95:1.2) -- (#1-4+0.3,#2-2+0.8)  (#1-4-0.3,#2-2+0.8) --  (#1-2*4+1.4,#2-2+0.8) (#1-2*4-1.4,#2-2+0.8) arc (90:270:0.8) -- (#1-4+2.8,#2-2-0.8) arc (90:0:1.2);
	\fill (#1,#2-4) circle (0.2);
} 

\newcommand{\encellsu}[2]{
	\fill (#1,#2) circle (0.2);	
	\draw[white,fill=white] (#1-2*4-1.1,#2-2-1-0.4) rectangle (#1+1.1,#2-2-1);		
	\draw[thick] (#1,#2) arc (0:-95:0.8) -- (#1-4+0.2,#2-2+1.2)  (#1-4-0.8,#2-2+1.2) --  (#1-2*4+1.4,#2-2+1.2) (#1-2*4-1.4,#2-2+1.2) arc (90:270:1.2) -- (#1-4+3.2,#2-2-1.2) arc (90:0:0.8);
	\fill (#1,#2-4) circle (0.2);
} 

\newcommand{\ellTii}[2]{
	\draw[lightgray,fill=lightgray] (#1-1,#2-2) rectangle (#1-0.5,#2+2);
	\fill (#1,#2+2) ellipse (1.4cm and 0.2cm);
	\fill (#1,#2-2) ellipse (1.4cm and 0.2cm);
	\draw[thick] (#1-1,#2+2) -- (#1-1,#2-2);
	\draw[thick] (#1-0.5,#2+2) -- (#1-0.5,#2-2);
	
	\slab{(#1+4,#2+3)}{$1$}
	\slab{(#1+8,#2+3)}{$i-1$}
	\slab{(#1+12,#2+3)}{$i$}
	\slab{(#1+16,#2+3)}{$i+1$}
	\fill (#1+4,#2+2) circle (0.2);
	\fill (#1+4,#2-2) circle (0.2);
	\fill (#1+8,#2+2) circle (0.2);
	\fill (#1+8,#2-2) circle (0.2);
	\fill (#1+12,#2+2) circle (0.2);
	\fill (#1+12,#2-2) circle (0.2);
	\fill (#1+16,#2+2) circle (0.2);
	\fill (#1+16,#2-2) circle (0.2);

	\draw[thick] (#1+4,#2+2)..controls +(0,-2) and +(0,+2) .. (#1+1,#2-2);
	\draw[thick] (#1+8,#2+2)..controls +(0,-2) and +(0,+2) .. (#1+4,#2-2);
	\draw[thick] (#1+12,#2+2)..controls +(0,-2) and +(0,+2) .. (#1+8,#2-2);
	\draw[thick] (#1+16,#2+2)..controls +(0,-2) and +(0,+2) .. (#1+12,#2-2);
	
	\fill[white] (#1+3.3,#2+0.6) circle (0.25);
	\draw[white,fill=white,rounded corners] (#1+5,#2-0.6) rectangle (#1+7.4,#2+1);
	\fill[white] (#1+11.1,#2+0.6) circle (0.25);
	\fill[white] (#1+14.65,#2+0.35) circle (0.25);
	
	\draw[thick] (#1+1,#2+2)..controls +(0,-1) and +(-1,0) .. (#1+2.5,#2+0.6) -- (#1+13,#2+0.6) ..controls +(2,0) and +(0,2) .. (#1+16,#2-2);
} 

\newcommand{\ellU}[2]{
	\draw[lightgray,fill=lightgray] (#1-1,#2-2) rectangle (#1-0.5,#2+2);
	\fill (#1,#2+2) ellipse (1.4cm and 0.2cm);
	\fill (#1,#2-2) ellipse (1.4cm and 0.2cm);
	\draw[thick] (#1-1,#2+2) -- (#1-1,#2-2);
	\draw[thick] (#1-0.5,#2+2) -- (#1-0.5,#2-2);
	
	\fill (#1+12,#2+2) circle (0.2);
	\fill (#1+12,#2-2) circle (0.2);
	\fill (#1+8,#2+2) circle (0.2);
	\fill (#1+8,#2-2) circle (0.2);
	
	\draw[thick] (#1+1,#2-2)..controls +(0,1) and +(-1,0) .. (#1+2.5,#2-1.2) -- (#1+3.6,#2-1.2) (#1+4.15,#2-1.2) -- (#1+7.8,#2-1.2) (#1+8.2,#2-1.2) --  (#1+9,#2-1.2) ..controls +(2.5,0) and +(0,-2.5) .. (#1+12,#2+2);
	\draw[thick] (#1+0.5,#2-2)..controls +(0,1.2) and +(-1.2,0) .. (#1+2.5,#2-0.6) -- (#1+3.1,#2-0.6) (#1+3.8,#2-0.6) -- (#1+5,#2-0.6) ..controls +(2.5,0) and +(0,-2.5) .. (#1+8,#2+2);
	\draw[thick] (#1,#2-2)..controls +(0,3) and +(0,-3)  .. (#1+4,#2+2);
	
	\fill[white] (#1+3.8,#2+1.2) circle (0.2);
	\fill[white] (#1+8,#2+1.2) circle (0.2);
	
	\fill[white] (#1+3.45,#2+0.6) circle (0.25);
	
	\fill[white] (#1+11.8,#2) circle (0.3);
	\fill[white] (#1+7.7,#2) circle (0.3);
	\fill[white] (#1+2,#2) circle (0.3);
	
	\draw[thick] (#1+1,#2+2)..controls +(0,-1) and +(-1,0) .. (#1+2.5,#2+1.2) -- (#1+9,#2+1.2) ..controls +(2.5,0) and +(0,2.5) .. (#1+12,#2-2);
	\draw[thick] (#1+0.5,#2+2)..controls +(0,-1.2) and +(-1.2,0) .. (#1+2.5,#2+0.6) -- (#1+5,#2+0.6) ..controls +(2.5,0) and +(0,2.5) .. (#1+8,#2-2);
	\draw[thick] (#1,#2+2)..controls +(0,-3) and +(0,3) .. (#1+4,#2-2);
	\fill (#1+4,#2+2) circle (0.2);
\fill (#1+4,#2-2) circle (0.2);
} 
                        
\title{\bf Fused Hecke algebra and one-boundary algebras}

\renewcommand*{\Affilfont}{\normalsize\small}
\author[1]{Lo\"ic Poulain d'Andecy}
\author[2]{Meri Zaimi\vspace{.5em}}

\affil[1]{Laboratoire de math\'ematiques de Reims UMR 9008, Universit\'e de Reims Champagne-Ardenne,  \newline\vspace{.9em} 
Moulin de la Housse BP 1039, 51100 Reims, France.}

\affil[2]{Centre de Recherches Math\'ematiques, Universit\'e de Montr\'eal, \newline\vspace{.9em}
P.O. Box 6128, Centre-ville Station, Montr\'eal (Qu\'ebec), H3C 3J7, Canada.}

{
	\makeatletter
	\renewcommand\AB@affilsepx{: \protect\Affilfont}
	\makeatother
	\affil[ ]{E-mail addresses}
	\makeatletter
	\renewcommand\AB@affilsepx{, \protect\Affilfont}
	\makeatother
	\affil[1]{loic.poulain-dandecy@univ-reims.fr}
	\affil[2]{meri.zaimi@umontreal.ca}
}

\begin{document}

\date{} 
\maketitle

\begin{abstract}
This paper gives an algebraic presentation of the fused Hecke algebra which describes the centraliser of tensor products of the $U_q(gl_N)$-representation labelled by a one-row partition of any size with vector representations. It is obtained through a detailed study of a new algebra that we call the symmetric one-boundary Hecke algebra. In particular, we prove that the symmetric one-boundary Hecke algebra is free over a ring of Laurent polynomials in three variables and we provide a basis indexed by a certain subset of signed permutations. We show how the symmetric one-boundary Hecke algebra admits the one-boundary Temperley--Lieb algebra as a quotient, and we also describe a basis of this latter algebra combinatorially in terms of signed permutations with avoiding patterns. The quotients corresponding to any value of $N$ in $gl_N$ (the Temperley--Lieb one corresponds to $N=2$) are also introduced. Finally, we obtain the fused Hecke algebra, and in turn the centralisers for any value of $N$, by specialising and quotienting the symmetric one-boundary Hecke algebra. In particular, this generalises to the Hecke case the description of the so-called boundary seam algebra, which is then obtained (taking $N=2$) as a quotient of the fused Hecke algebra.
\end{abstract}

\section{Introduction}

The usual Hecke algebra $H_n(q)$ appears in the quantum Schur--Weyl duality \cite{Jimbo} describing the centralisers of tensor powers of the vector representation of the quantum group $U_q(gl_N)$. If we denote $L^N$ the vector representation of dimension $N$ of $U_q(gl_N)$, there is a surjective morphism:
\[H_n\to\text{End}_{U_q(gl_N)}((L^N)^{\otimes n})\ .\]
The Hecke algebra $H_n$ does not depend on $N$ and plays its role for $gl_N$ for any $N$. The dependence on $N$ of the centraliser appears in the description of the kernel of the above map. Indeed, for a given $N$, the centraliser of $U_q(gl_N)$ is isomorphic to the quotient of $H_n$ by the $q$-antisymmetriser on $N+1$ points. This $q$-antisymmetriser is a minimal central idempotent of $H_{N+1}$ (the quotient is trivial if $n\leq N$) and generates the kernel for any $n\geq N+1$. In particular, for $N=2$, the resulting algebra is the well-known Temperley--Lieb algebra. 

We would like to generalise the whole picture for tensor products of more general representations of $U_q(gl_N)$. The fused Hecke algebra was introduced in \cite{CP} for this purpose. For $\vec{k}=(k_1,\dots,k_n)$ a vector of positive integers, we have a surjective morphism:
\begin{equation}\label{cent-intro-gen}
H_{\vec{k},n}\to\text{End}_{U_q(gl_N)}(L^N_{(k_1)}\otimes L^N_{(k_2)}\otimes \dots \otimes L_{(k_n)}^N)\ ,
\end{equation}
where $L_{(k)}^N$ is the $k$-th $q$-symmetrised power of $L^N$ (in other words, the irreducible representation of $U_q(gl_N)$ indexed by the one-row partition of size $k$) and $H_{\vec{k},n}$ is called the fused Hecke algebra. Again, the algebra $H_{\vec{k},n}$ does not depend on $N$ and for large $N$ is exactly the centraliser. The dependence of the centralisers on $N$ appears in the kernel of the surjective map, of which an explicit description is conjectured in \cite{CP} and proved in some cases, including the ones we will study in this paper. For $N=2$, the centralisers can be called the fused Temperley--Lieb algebras in our terminology. They appear in several recent works and are also known as seam algebras, valenced Temperley--Lieb algebra or Jones--Wenzl algebras \cite{ALZ,FP,Spe}.

There is no known presentation by generators and relations for the fused Hecke algebra $H_{\vec{k},n}$ in general and in turn no known presentation for the $U_q(gl_N)$-centralisers (see \cite{FP} for a study of this question for the Jones--Wenzl algebras, that is, for $U_q(gl_2)$-centralisers),

\medskip
This paper is concerned with the case where only the first representation in \eqref{cent-intro-gen} is fused. Namely we fix $k>0$ and we denote by $H_{k,n}$ the fused Hecke algebra corresponding to the following centraliser:
\begin{equation}\label{cent-intro}
H_{k,n}\to\text{End}_{U_q(gl_N)}(L^N_{(k)}\otimes (L^N)^{\otimes n})\ .
\end{equation}
This situation is commonly referred to as the one-boundary case. Such one-boundary centraliser algebras have been studied especially for $U_q(gl_2)$, and also often with an infinite-dimensional module (like a Verma module) as the first factor; see \cite{CGS,ILZ,LV} for recent works.  For $N=2$, the one-boundary case of the fused Temperley--Lieb algebra is referred to as the boundary seam algebra \cite{LS,L-L,MRR} and is a quotient of the one-boundary Temperley--Lieb algebra or blob algebra \cite{tD,MS,NRdG}. The presentation given in \cite{MRR}, even if not explicitly stated this way, can be seen as a description of the centraliser \eqref{cent-intro} in the $gl_2$ case. 

\medskip
The first main goal of this paper is to give an algebraic presentation of the fused Hecke algebra $H_{k,n}$ and of its quotients corresponding to the centralisers for any $N$. In particular, we obtain the boundary seam algebra ($N=2$) explicitly as a quotient of the fused Hecke algebra $H_{k,n}$, and provide its generalisation for any $N>2$.

The second main purpose of this paper is the introduction of a new algebra, which we denote $\cA_n$ and call the symmetric one-boundary Hecke algebra. Roughly speaking, this 3-parameter algebra $\cA_n$ allows to interpolate between all algebras $H_{k,n}$ for varying $k$. The word \emph{symmetric} is meant to recall the fact that the representations allowed at the boundary are the $q$-symmetric powers. The algebra $\cA_n$ admits as a quotient the 3-parameter one-boundary Temperley--Lieb algebra, which we denote here $\cC_{n,2}$, and we also define naturally as quotients of $\cA_n$ the generalisations $\cC_{n,N}$ corresponding to $gl_N$ for any $N>2$.

\medskip
We now describe more precisely, step by step, the algebras involved in the paper. It is well-known, see for example \cite{OR}, that the one-boundary centraliser in \eqref{cent-intro} is a quotient of the affine Hecke algebra. Moreover, since the partition $(k)$ made of a line of $k$ boxes has only two addable nodes, this quotient factors through a cyclotomic quotient of level 2, in other words through a Hecke algebra of type B. This is the starting point of our constructions.

\paragraph{The starting point $\cH{n}$.} We start with the type B Hecke algebra $\cH{n}$ defined over the ring $\bC[q^{\pm1},\alpha_1^{\pm1},\alpha_2^{\pm2}]$ with three indeterminates. The indeterminates $\alpha_1$ and $\alpha_2$ correspond to the eigenvalues of the boundary, or type B, generator, while the eigenvalues of the other generators are $q$ and $-q^{-1}$.

The algebra $\cH{n}$ has a standard basis indexed by the signed permutations and we have a good understanding of the representation theory over the field of fractions $\bC(q,\alpha_1,\alpha_2)$. Namely, the algebra is semisimple and the irreducible representations are indexed by bipartitions of $n$. Among these irreducible representations, four of them are of dimension 1 and they correspond to the following bipartitions:
\[\begin{array}{lllllll}
(\Box\dots\Box\,,\,\emptyset) &\quad& (\begin{array}{l}\Box\\[-0.8em] \vdots\\[-0.4em] \Box\end{array}\,,\,\emptyset) & \quad &
(\emptyset\,,\,\Box\dots\Box) &\quad& (\emptyset\,,\,\begin{array}{l}\Box\\[-0.8em] \vdots\\[-0.4em] \Box\end{array})\\

(q,\alpha_1) && (-q^{-1},\alpha_1) && (q,\alpha_2) && (-q^{-1},\alpha_2)
\end{array}\ \ \ \ \ .\]
Each of these one-dimensional representations corresponds to a choice of eigenvalues for the generators, as indicated above. Moreover, explicit expressions for the minimal central idempotents corresponding to these representations are known (see Section \ref{subsec-idem}). These idempotents live in the algebra $\cH{n}$ extended over the field of fractions $\bC(q,\alpha_1,\alpha_2)$, and by simply removing the denominators in these explicit expressions, we obtain central quasi-idempotents in the algebra $\cH{n}$ well-defined over the ring of polynomials. These quasi-idempotents are denoted $E_n^{(x,y)}$, where $(x,y)$ are the corresponding eigenvalues, and will be crucial to all our constructions. Indeed all algebras involved in the paper are obtained by quotienting by some of these quasi-idempotents, as summarised in the following picture:
\begin{center}
\begin{tikzpicture}[scale=1]
\node at (-1.5,0) {$\cH{n}$};

\draw[->] (-1.5,-0.4) -- (-1.5,-1.8);

\node at (1.5,-1) {\footnotesize $\left(\begin{array}{l}\alpha_1=q^{-2}\\ \alpha_2=q^{2k}\end{array}\right)$};
\node at (-2.5,-1) {\footnotesize $E_2^{(-q^{-1},\alpha_2)}$};

\node at (-1.5,-2.2) {$\cA_n$};
\node at (1.3,-2.2) {$\cA_n^{(k)}$};
\node at (-1.5,-4.4) {$\cC_{n,N}$};
\node at (1.3,-4.4) {$\cC_{n,N}^{(k)}$};

\draw[->] (-1.5,-2.6) -- (-1.5,-4);
\node at (-2.5,-3.2) {\footnotesize $E_N^{(-q^{-1},\alpha_1)}$};

\draw[->] (1.3,-2.6) -- (1.3,-4);
\node at (2.3,-3.2) {\footnotesize $E_N^{(-q^{-1},\alpha_1)}$};

\draw[->,densely dashed] (-1,-2.2) -- (0.8,-2.2);
\node at (-0.1,-1.8) {\footnotesize $E_{k+1}^{(q,\alpha_1)}$};

\draw[->,densely dashed] (-1,-4.4) -- (0.8,-4.4);
\node at (-0.1,-4.8) {\footnotesize $E_{k+1}^{(q,\alpha_1)}$};
\end{tikzpicture}
\end{center}
The full lines represent genuine quotients, while dashed lines represent quotients combined with a specialisation of the parameters $\alpha_1,\alpha_2$ as indicated in the diagram. We briefly detail each step of the diagram.

\paragraph{The algebra $\cA_n$.} The symmetric one-boundary Hecke algebra $\cA_n$ is obtained from $\cH{n}$ by quotienting out the quasi-idempotent $E_2^{(-q^{-1},\alpha_2)}$. This quotient is the main object of study of Section \ref{sec-A}. Quite naturally from its definition, the irreducible representations of the algebra $\cA_n$ over the field of fractions are indexed by bipartitions with a one-row partition as the second component. Our first main result is that the algebra $\cA_n$ is free over $\bC[q^{\pm1},\alpha_1^{\pm1},\alpha_2^{\pm2}]$ and we provide a basis in terms of signed permutations with avoiding patterns. We conclude this section with a technical fact, namely, that some of the remaining central quasi-idempotents $E_n^{(x,y)}$ can be renormalised in $\cA_n$ while still being well-defined over $\bC[q^{\pm1},\alpha_1^{\pm1},\alpha_2^{\pm2}]$. This will be important for what follows in order for the subsequently defined quotients to behave well.

\paragraph{The algebras $\cC_{n,N}$.} Then in Section \ref{sec:C}, for $N>1$, the symmetric one-boundary $N$-centraliser algebras $\cC_{n,N}$ are defined by further quotienting $\cA_n$ by one of the remaining (and renormalised) quasi-idempotent when $n \geq N$ (as well as the usual $q$-antisymmetriser on $N+1$ points). The relevant quasi-idempotent is indicated in the diagram above and this definition leads easily to the description of the representation theory over the field of fractions: the irreducible representations are now indexed by bipartitions $(\lambda,\mu)$ where $\mu$ is a one-row partition and $\lambda$ has strictly less than $N$ rows. 

The algebras $\cC_{n,N}$ include the one-boundary Temperley--Lieb algebra, which is the case $N=2$, and provide its natural generalisation for general $N$. The name comes from the fact that the algebra $\cC_{n,N}$ somehow interpolates the $U_q(gl_N)$-centralisers in \eqref{cent-intro} for fixed $N$ and varying $k$. The case $N=2$ is examined in more details and again, we show that $\cC_{n,2}$ is free over the ring $\bC[q^{\pm1},\alpha_1^{\pm1},\alpha_2^{\pm2}]$ with a basis also given in terms of signed permutations with certain avoiding patterns (this is where the renormalisation mentioned above is important). This description can be seen as a one-boundary generalisation of the description of the usual Temperley--Lieb algebra in terms of usual permutations with avoiding patterns.

\paragraph{The algebras $\cA_n^{(k)}$.} Section \ref{sec:fusedH} is mainly devoted to the algebraic description of the fused Hecke algebra $\fH$. For this purpose, we define the algebra $\cA_n^{(k)}$ as a specialisation of $\cA_n$ followed by a quotient by another (renormalized) quasi-idempotent for $n\geq k+1$. The specialisation replaces $\alpha_1$ and $\alpha_2$ by the eigenvalues of the boundary generator in the fused Hecke algebra. Their values are indicated in the diagram above together with the relevant quasi-idempotent. The main result of the section is that the algebras $\cA_n^{(k)}$ and $\fH$ are isomorphic. This leads to a presentation of the fused Hecke algebra in terms of generators and relations. Again, a basis of $\cA_n^{(k)}$ in terms of signed permutations with avoiding patterns is provided.    

\paragraph{The algebras $\cC_{n,N}^{(k)}$.} Lastly, in Section \ref{sec-cent}, the algebras $\cC_{n,N}^{(k)}$ are defined by naturally completing the square of the picture above. Namely, they are defined either as specialisations and quotients of $\cC_{n,N}$, or equivalently as quotients of $\cA_n^{(k)}$. The algebras $\cC_{n,N}^{(k)}$ are shown to be isomorphic to the $U_q(gl_N)$-centraliser in \eqref{cent-intro} and this provides an algebraic description of the centraliser. We show that we have reobtained naturally with $\cC_{n,2}^{(k)}$ the boundary seam algebra of \cite{MRR} and thus the algebras $\cC_{n,N}^{(k)}$ can be seen as the $gl_N$-generalisations. Following our results, a natural definition of the algebra $\cC_{n,2}^{(k)}$ over $\mathbb{C}[q^{\pm 1}]$ is given (and here we differ from \cite{MRR}) and it is shown to be free over $\mathbb{C}[q^{\pm 1}]$ with a basis given explicitly.

%
%

\paragraph{Acknowledgements.} The authors thank Yvan Saint-Aubin for many interesting discussions. The first author is supported by Agence Nationale de la Recherche Projet AHA ANR-18-CE40-0001 and the international research project AAPT of the CNRS. The second author is grateful to the LMR for its hospitality and is supported by the international research project AAPT of the CNRS. The second author also holds an Alexander-Graham-Bell scholarship from the Natural Sciences and Engineering Research Council of Canada (NSERC).

\setcounter{tocdepth}{1}
\tableofcontents

\section{The symmetric one-boundary Hecke algebra}\label{sec-A}

We will use the notations:
\[[r]_x=\frac{x^r-x^{-r}}{x-x^{-1}}=x^{r-1}+x^{r-3}+\dots+x^{1-r}\ \ \ \text{and}\ \ \ [r]_x!=[2]_x\dots [r]_x
\ .\]
We will be working with the ring $R=\bC[q^{\pm 1},\alpha_1^{\pm 1},\alpha_2^{\pm 1}]$, where $q,\alpha_1,\alpha_2$ are indeterminates, and with its field of fractions $F=\bC(q,\alpha_1,\alpha_2)$.

\subsection{The cyclotomic Hecke algebra of level 2}\label{subsec-cyclo}

Let $n\geq0$. We define the algebra $\cH{n}$ as the algebra over $R$ with generators $g_i$ for $i=0,1,\dots,n-1$ and defining relations
\begin{alignat}{2}
	&g_ig_{i+1}g_i=g_{i+1}g_ig_{i+1}, &&\quad 1\leq i \leq n-2, \label{eq:cycHrel1} \\	
	&g_0g_1g_0g_1 = g_1g_0g_1g_0, && \label{eq:cycHrel2}\\
	&g_ig_j=g_jg_i, && \quad |i-j|\geq 2, \label{eq:cycHrel3}\\
	&(g_i-q)(g_i+q^{-1})=0, && \quad 1\leq i \leq n-1, \label{eq:cycHrel4}\\
	&(g_0-\alpha_1)(g_0-\alpha_2)=0. && \label{eq:cycHrel5}
\end{alignat}
By convention, $\cH{0}=R$. The algebra is a quotient of the affine Hecke algebra of type $A$ by the last relation. It is called a cyclotomic Hecke algebra of level 2 since this last relation is a quadratic characteristic relation for $g_0$. The algebra $\cH{n}$ is also the Hecke algebra associated to a Coxeter group of type $B_n$. For what is recalled below about $\cH{n}$, see \emph{e.g.} \cite{GP}.

The algebra $\cH{n}$ is free as an $R$-module and has a basis labelled by the elements of the Coxeter group of type $B_n$. We will abuse notations and denote $B_n$ this Coxeter group. Its elements can be viewed as signed permutations, that is, those permutations $\omega$ on the set $\{-n,-n+1,\dots,-1,1,2,\dots,n\}$ such that $\omega(-i)=-\omega(i)$ for all $i \in \{1,2,\dots,n\}$. We can represent the elements of $B_n$ by words $b=b_1b_2\dots b_n$ where each of the numbers $1,2,\dots,n$ appears once and is possibly barred (see \textit{e.g.}\ \cite{Sim}). In this representation, $b_i$ is the image of $i$ by $\omega$, where the bar notation is understood as a negative sign. The group $B_n$ contains $n!2^n$ elements. 

We denote by $s_i$ the transposition of $i$ and $i+1$ in $B_n$ (which thus also transposes $-i$ and $-(i+1)$), and by $s_0$ the transposition of $-1$ and $1$. The group $B_n$ is generated by $s_i$ with $i=0,\dots,n-1$. For an element $\omega\in B_n$, we write it as a reduced expression, that is, as a product $s_{i_1}\dots s_{i_k}$ with minimal $k$. We set $g_{\omega}=g_{i_1}\dots g_{i_k}$ in $\cH{n}$. The element $g_{\omega}$ does not depend on the choice of the reduced expression, and the set $\{g_{\omega}\}_{\omega\in B_n}$ forms an $R$-basis of $\cH{n}$. A standard choice of reduced forms leads to an explicit expression for the basis as the following product of sets:
\begin{equation}
\left\{\begin{array}{c} 
1,\\ g_0 
\end{array}\right\} 
\cdot 
\left\{\begin{array}{c} 
1,\\ g_1, \\ g_1g_0, \\ g_1g_0g_1 \end{array}\right\}
\cdot 
\left\{\begin{array}{c} 
1,\\ g_2, \\ g_2g_1, \\ g_2g_1g_0, \\ g_2g_1g_0g_1, \\ g_2g_1g_0g_1g_2 \end{array}\right\}
\cdot\ \ldots\ \cdot 
\left\{\begin{array}{c} 1,\\ g_{n-1}, \\ \vdots \\ g_{n-1}\dots g_1g_0,\\ g_{n-1}\dots g_1g_0g_1, \\ \vdots \\ g_{n-1}\dots g_1g_0g_1\dots g_{n-1} \end{array}\right\}. \label{eq:basecH}
\end{equation} 
Introducing the following notation for $0\leq m\leq n$:
\[[n,m]=g_n\dots g_{m+1}g_m\ \ \ \ \ \text{and}\ \ \ \ [n,-m]=g_n\dots g_1g_0g_1\dots g_m\,,\]
the basis elements can be written as:
\[[n_1,m_1][n_2,m_2]\dots [n_k,m_k]\ \ \ \ \ \ \text{with}\ 0\leq n_1<n_2<\dots<n_k\leq n-1\ \ \text{and}\ \ |m_i|\leq n_i\ .\]
The algebra $\cH{n-1}$ is naturally a subalgebra of $\cH{n}$, the one generated by $g_0,\dots,g_{n-2}$, where elements of $B_{n-1}$ are naturally identified with elements of $B_n$ leaving invariant the letter $n$.

\paragraph{Hecke algebra of type A.} The algebra generated by $g_1,\dots,g_{n-1}$ with defining relations \eqref{eq:cycHrel1}, \eqref{eq:cycHrel3}, \eqref{eq:cycHrel4} is the usual Hecke algebra $H_n$ of type $A$, associated to the symmetric group $\mathfrak{S}_n$ on $n$ letters. It is naturally identified as the subalgebra of $\cH{n}$ generated by $g_1,\dots,g_{n-1}$, and a basis of $H_n$ is the subset $\{g_\omega\}_{\omega\in \fS_n}$, when the symmetric group is naturally considered as a subgroup of $B_n$.

The basis of $H_n$ is made of those elements in \eqref{eq:basecH} which do not contain $g_0$. The basis \eqref{eq:basecH} is well adapted to the inclusion $B_{n-1}\subset B_n$. There is a different way to produce a basis of $\cH{n}$ adapted to the inclusion $\fS_n\subset B_n$, which is the set of elements:
\begin{equation}\label{eq:basecH2}
g_{\omega}\cdot g_0g_1\dots g_{i_1}\cdot \ldots \cdot g_0g_1\dots g_{i_k}\ ,\ \ \ \ \omega\in\fS_n\ \text{and}\ n-1\geq i_1>\dots >i_k\geq 0\ .
\end{equation}

\subsection{Central quasi-idempotents in $\cH{n}$}\label{subsec-idem}

For $i=0,1,\dots , n-1$ and $\omega \in B_n$, we denote $\ell(\omega)$ the length of $\omega$, which is the number of generators appearing in any reduced expression of $\omega$. We denote $\ell_0(\omega)$ the number of times that the generator $s_0$ appears in a reduced expression for $\omega$. This does not depend on the reduced expression since all braid relations in $\cH{n}$ are homogeneous in $g_0$.

\paragraph{$q$-symmetriser and $q$-antisymmetriser in the Hecke algebra.} First we discuss the standard quasi-idempotents in the usual Hecke agebra $H_n$ generated by $g_1,\dots,g_{n-1}$. Let $x\in \{ q, -q^{-1} \}$ and set:
\begin{equation}\label{eq:symHecke}
\Lambda^{x}_n(g_1,\dots,g_{n-1})=\sum_{\omega\in\fS_{n}} x^{\ell(\omega)}g_\omega\ .
\end{equation}
By convention $\Lambda^{x}_1=1$. Using the basis in \eqref{eq:basecH} without the elements containing $g_0$, we find the recursive formula:
\[
\Lambda^{x}_n(g_1,\dots,g_{n-1})=\Lambda^{x}_{n-1}(g_1,\dots,g_{n-2})(1+x g_{n-1}+\dots+x^{n-1}g_{n-1}\dots g_1)\ .
\]
It is well-known and easy to check (see the proof of Proposition \ref{propE} below) that:
\[\Lambda^{x}_n(g_1,\dots,g_{n-1})g_i=g_i\Lambda^{x}_n(g_1,\dots,g_{n-1}) =x\Lambda^{x}_n(g_1,\dots,g_{n-1}), \ \ \quad i=1,\dots, n-1\,.\]
It follows that these two elements are central in $H_n$ and are quasi-idempotents, namely,
\[\bigl(\Lambda^{x}_n(g_1,\dots,g_{n-1})\bigr)^2=q^{\pm\frac{n(n-1)}{2}}[n]_{q}!\,\Lambda^{x}_n(g_1,\dots,g_{n-1})\ \ \ \ \ \text{when $x=\pm q^{\pm1}$}.\]
To find the coefficient, one needs to replace each $g_i$ by $x$ in the formula for the quasi-idempotents. This is easily done using their recursive formula. We refer to the element with $x=q$ as the (unnormalised) $q$-symmetriser of $H_n$ and to the element with $x=-q^{-1}$ as the (unnormalised) $q$-antisymmetriser of $H_n$.

\paragraph{The four quasi-idempotents in $\cH{n}$.} Let $x\in \{ q, -q^{-1} \}$ and $b\in\{1,2\}$. In what follows, it will be convenient to consider the indices of $\alpha_1,\alpha_2$ modulo 2, so that $\alpha_{b+1}=\alpha_1$ when $b=2$. We define:
\begin{equation}\label{defE}
	E^{(x,\alpha_b)}_n := \sum_{\omega \in B_n} z_\omega g_\omega, \qquad \text{where } z_\omega := x^{\ell(\omega)-\ell_0(\omega)}(-\alpha_{b+1}^{-1})^{\ell_0(\omega)}\ .
\end{equation}
By convention, $E^{(x,\alpha_b)}_0=1$. Using the standard basis in \eqref{eq:basecH}, a recursive formula for these elements is:
\begin{equation}
	E^{(x,\alpha_b)}_n = E^{(x,\alpha_b)}_{n-1}(1+\sum_{i=1}^{n-1}x^{n-i}g_{n-1}\dots g_{i}-x^{n-1}\alpha_{b+1}^{-1}g_{n-1}\dots g_1g_0(1+\sum_{i=1}^{n-1}x^ig_1\dots g_i) ). \label{eq:Erec} 
\end{equation} 
Using the basis \eqref{eq:basecH2} adapted to the embedding $H_n\subset \cH{n}$, we also have:
\begin{equation}
E_n^{(x,\alpha_b)}=\Lambda^{x}_n(g_1,\dots,g_{n-1})\cdot (1-x^{n-1}\alpha_{b+1}^{-1}g_0g_1\dots g_{n-1})\ldots (1-x\alpha_{b+1}^{-1}g_0g_1)(1-\alpha_{b+1}^{-1}g_0)\,. \label{eq:ELambda} 
\end{equation}
Explicit examples for small $n$ are
\begin{align*}
	&E^{(x,\alpha_b)}_1 = 1-\alpha_{b+1}g_0\,,\\
	&E^{(x,\alpha_b)}_2=1+xg_1-\alpha_{b+1}g_0-x\alpha_{b+1}(g_1g_0+g_0g_1)-x^2\alpha_{b+1}g_1g_0g_1+x\alpha_{b+1}^2g_0g_1g_0+x^2\alpha_{b+1}^2g_0g_1g_0g_1\ . \label{eq:E}
\end{align*} 
We recall the important facts about the elements $E^{(x,\alpha_b)}_n$, implying in particular that they are central quasi-idempotents of $\cH{n}$.
\begin{prop}\label{propE}
Let $x=\pm q^{\pm1}$ and $b\in\{1,2\}$. We have:
\begin{equation}\label{eq:proprE}
E^{(x,\alpha_b)}_n g_0=g_0 E^{(x,\alpha_b)}_n =\alpha_bE^{(x,\alpha_b)}_n\ \ \ \text{and}\ \ \ E^{(x,\alpha_b)}_n g_i=g_iE^{(x,\alpha_b)}_n =xE^{(x,\alpha_b)}_n, \ \quad i=1,\dots, n-1\,,
\end{equation}
and
\begin{align}
	\bigl(E^{(x,\alpha_b)}_n\bigr)^2 =P_n(x,\alpha_b)E^{(x,\alpha_b)}_n\,,\ \ \ \text{where}\ P_n(x,\alpha_b)= q^{\pm\frac{n(n-1)}{2}}[n]_{q}!\prod_{i=0}^{n-1}\left(1-q^{\pm 2i}\frac{\alpha_b}{\alpha_{b+1}}\right)\ . \label{eq:Ensquare}
\end{align}	 
\end{prop}
\begin{proof}
For any $\omega\in B_n$ and any $i \in \{0,1,\dots,n-1\}$, we have $\ell(s_i\omega) = \ell(\omega)\pm 1$. If $\ell(s_i\omega)>\ell(\omega)$, then we have $g_{s_iw}=g_ig_w$. Therefore we can write:
\begin{equation}
	E^{(x,\alpha_b)}_n = \sum_{\genfrac{}{}{0pt}{}{\omega \in B_n}{\ell(s_i\omega)>\ell(\omega)}} (z_\omega g_\omega + z_{s_i\omega} g_ig_{\omega}).
\end{equation}
Note that under the hypothesis that $\ell(s_i\omega)>\ell(\omega)$, we must have $z_{s_i\omega} = x z_{\omega}$ if $1\leq i \leq n-1$ and $z_{s_i\omega} = -\alpha_{b+1}^{-1} z_{\omega}$ if $i=0$. The defining relations \eqref{eq:cycHrel4} and \eqref{eq:cycHrel5} imply that
\begin{equation}
	(g_0-\alpha_b)(1-\alpha_{b+1}^{-1}g_0)=0\ \ \ \ \text{and}\ \ \ \ \ (g_i-x)(g_i+x^{-1})=0\ \ (i=1,\dots,n-1).
\end{equation}
Using the previous equations, we have, for example if $1\leq i \leq n-1$, that
\begin{equation}
	(g_i-x)E^{(x,\alpha_b)}_n
	= \sum_{\genfrac{}{}{0pt}{}{\omega \in B_n}{\ell(s_i\omega)>\ell(\omega)}} (g_i-x)(z_\omega g_\omega + xz_{\omega} g_ig_{\omega}) 
	= \sum_{\genfrac{}{}{0pt}{}{\omega \in B_n}{\ell(s_i\omega)>\ell(\omega)}} xz_\omega(g_i-x)( x^{-1} + g_i)g_\omega=0.
\end{equation}
The case $i=0$ is similarly done. Moreover, similar arguments can be used when considering instead the product $E^{(x,\alpha_b)}_ng_i$ for $0 \leq i \leq n-1$. This proves \eqref{eq:proprE}.

Now the coefficient $P_n(x,\alpha_b)$ in \eqref{eq:Ensquare} is found by replacing $g_0$ by $\alpha_b$ and the other $g_i$'s by $x$ in the formula for $E^{(x,\alpha_b)}_n$. The given formula for $P_n(x,\alpha_b)$ follows then easily from \eqref{eq:ELambda}.
\end{proof}

\begin{rem}
Over the field of fractions $F$, or in a specialisation such that $P_{n-1}(x,\alpha_b)\neq0$, we have also the following recursive formula:
\begin{equation}
E^{(x,\alpha_b)}_n = E^{(x,\alpha_b)}_{n-1}+x\frac{E^{(x,\alpha_b)}_{n-1}g_{n-1}E^{(x,\alpha_b)}_{n-1}}{P_ {n-2}(x,\alpha_b)}-x^{2(n-1)}\alpha_{b+1}^{-1}\frac{E^{(x,\alpha_b)}_{n-1}g_{n-1}\dots g_1g_0g_1\dots g_{n-1}E^{(x,\alpha_b)}_{n-1}}{P_ {n-1}(x,\alpha_b)}\ . \label{eq:Erec2}
\end{equation}
\end{rem}

\subsection{The symmetric one-boundary Hecke algebra $\cA_n$}

We define below the main object of this section that we call the symmetric one-boundary Hecke algebra.
\begin{defi}
Let $n\geq 0$. We define the symmetric one-boundary Hecke algebra $\cA_n$ as the algebra over $R$ which is the quotient of $\cH{n}$ by the relation:
\begin{equation}
	E_2^{(-q^{-1},\alpha_2)}=0\ . \label{eq:relidem1}
\end{equation}
It is understood that $\cA_n=\cH{n}$ if $n=0,1$.
\end{defi}
Using the explicit expression of $E_2^{(-q^{-1},\alpha_2)}$, this is equivalent to imposing the following relation:
\begin{equation}
	g_0g_1g_0g_1=-q^{2}\alpha_1^{2}+q\alpha_1^{2}g_1+q^{2}\alpha_1g_0-q\alpha_1(g_1g_0+g_0g_1) +\alpha_1 g_1g_0g_1+qg_0g_1g_0\ . \label{eq:relquotient}
\end{equation}

\paragraph{Semisimple representation theory.} In this paragraph, we extend the algebras $\cH{n}$ and $\cA_n$ over the field of fractions $F$, and denote them $F\cH{n}$ and $F\cA_n$ to avoid any confusion. The representation theory of $F\cH{n}$ is well-known \cite{AK,Ho,OP11}, and can be described in terms of bipartitions and Young tableaux.

A partition $\lambda$ of $n$, denoted $\lambda\vdash n$, is a decreasing sequence of  positive integers $\lambda=(\lambda_1,\dots,\lambda_k)$ with $\lambda_1\geq\lambda_2\geq\dots\geq\lambda_k \geq 1$ such that $\lambda_1+\dots+\lambda_k=n$. We also say that $\lambda$ is a partition {\em of size} $n$ and denote $|\lambda|=n$. We identify partitions with their Young diagrams: the Young diagram of $\lambda$ is a left-justified array of rows of boxes such that the $j$-th row (we count from top to bottom) contains  $\lambda_j$ boxes. The number of non-empty rows is the length $\ell(\lambda)$ of $\lambda$. By convention, the empty set $\emptyset$ is the only partition of $n=0$.

A standard tableau of shape $\lambda$  is a bijective filling of the boxes of $\lambda$ by numbers $1,\dots,n$ such that the entries strictly increase along any row and down 
 any column of the diagram. We denote by $d_{\lambda}$ the number of standard tableaux of shape $\lambda$. From the representation theory of the symmetric group, or from the Robinson--Schensted correspondence, we have:
 \begin{equation}\label{dimfactn}
 \sum_{\lambda\vdash n}d_{\lambda}^2=n!\ .
 \end{equation}

A bipartition of size $n$ is a pair of partition $(\lambda,\mu)$ such that $|\lambda|+|\mu|=n$. We denote $Par_2(n)$ the set of bipartitions of $n$. A standard tableau of shape $(\lambda,\mu)$ is a bijective filling of the boxes of $\lambda$ and $\mu$ by the numbers $1,\dots,n$ such that the entries strictly increase along any row and down 
 any column of the two diagrams. The number of standard tableaux of shape $(\lambda,\mu)$ is easily seen to be:
 \begin{equation}\label{dimbipart}
 d_{\lambda,\mu}=\binom{n}{|\lambda|}d_{\lambda} d_{\mu}\ .
 \end{equation}

The set of irreducible representations of $F\cH{n}$ is indexed by the bipartitions of size $n$, and we will denote $V_{(\lambda,\mu)}$ the irreducible representations indexed by $(\lambda,\mu)\in Par_2(n)$ so that:
\[Irr(F\cH{n})=\{V_{(\lambda,\mu)}\ |\ (\lambda,\mu)\in Par_2(n)\}\ .\]
There are four one-dimensional representations of $F\cH{n}$ for $n\geq 2$ and the parametrisation is made such that they correspond to the following bipartitions of $n$, with the given corresponding values, respectively, of $g_0$ and of $g_i$, $i=1,\dots,n-1$:
\begin{equation}\label{1dimrep}\begin{array}{lllllll}
(\Box\dots\Box\,,\,\emptyset) &\quad& (\begin{array}{l}\Box\\[-0.8em] \vdots\\[-0.4em] \Box\end{array}\,,\,\emptyset) & \quad &
(\emptyset\,,\,\Box\dots\Box) &\quad& (\emptyset\,,\,\begin{array}{l}\Box\\[-0.8em] \vdots\\[-0.4em] \Box\end{array})
\\
\ \ g_0\mapsto \alpha_1 && g_0\mapsto \alpha_1 && \ \ g_0\mapsto \alpha_2 && g_0\mapsto \alpha_2 \\
\ \ g_i\mapsto q        && g_i\mapsto -q^{-1} && \ \ g_i\mapsto q && g_i\mapsto -q^{-1}\\
\end{array}\end{equation}
Moreover, the branching rules expressing the restriction from $\cH{n}$ to $\cH{n-1}$ are given by inclusion of bipartitions (or more precisely, of their Young diagrams), as shown in the begining of the Bratteli graph below. We refer to an appendix in \cite{CP} for a discussion of Bratteli diagrams and of quotients of semisimple algebras by central idempotents.

The parametrisation of the irreducible representations is uniquely fixed by these requirements, and the dimension of the irreducible representation $V_{(\lambda,\mu)}$ is the number of standard tableaux of shape $(\lambda,\mu)$:
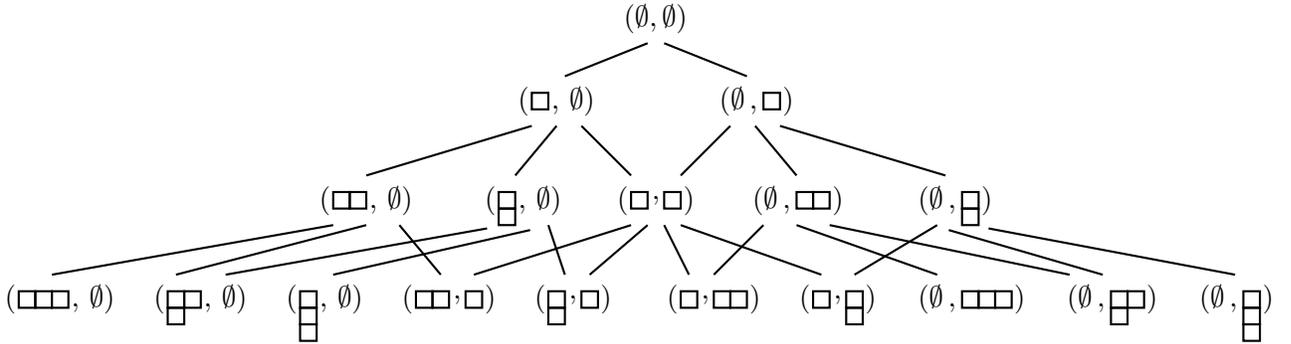
\[\dim V_{(\lambda,\mu)}=d_{(\lambda,\mu)}=\binom{n}{|\lambda|}d_{\lambda} d_{\mu}\ .\]

\begin{center}
\begin{tikzpicture}[scale=0.22]
\node at (0,0) {$(\emptyset,\emptyset)$};

\draw[thick] (-0.5,-1.5) -- (-5.5,-3.5);
\draw[thick] (0.5,-1.5) -- (5.5,-3.5);

\node at (-8,-5) {$($};\diag{-7.5}{-4.5}{1};\node at (-5,-5) {$,\,\emptyset)$};
\node at (5,-5) {$(\emptyset\,,$};\diag{6.5}{-4.5}{1};\node at (8,-5) {$)$};

\draw[thick] (-7.5,-6.5) -- (-17.5,-9.5);
\draw[thick] (-6,-6.5) -- (-8.5,-9.5);
\draw[thick] (-4.5,-6.5) -- (-1.5,-9.5);
\draw[thick] (4.5,-6.5) -- (1.5,-9.5);
\draw[thick] (6,-6.5) -- (8.5,-9.5);
\draw[thick] (7.5,-6.5) -- (17.5,-9.5);

\node at (-20,-11) {$($};\diag{-19.5}{-10.5}{2};\node at (-16,-11) {$,\,\emptyset)$};
\node at (-10,-11) {$($};\diagg{-9.5}{-10.5}{1}{1};\node at (-7,-11) {$,\,\emptyset)$};
\node at (-2,-11) {$($};\diag{-1.5}{-10.5}{1};\node at (0,-11) {$,$};\diag{0.5}{-10.5}{1};\node at (2,-11) {$)$};
\node at (7,-11) {$(\emptyset\,,$};\diag{8.5}{-10.5}{2};\node at (11,-11) {$)$};
\node at (17,-11) {$(\emptyset\,,$};\diagg{18.5}{-10.5}{1}{1};\node at (20,-11) {$)$};

\draw[thick] (-19.5,-12.5) -- (-36.5,-15.5);
\draw[thick] (-17.5,-12.5) -- (-29,-15.5);
\draw[thick] (-15.5,-12.5) -- (-13,-15.5);

\draw[thick] (-10.2,-12.7) -- (-26,-15.5);
\draw[thick] (-7.6,-12.8) -- (-19.5,-15.5);
\draw[thick] (-6.5,-12.5) -- (-5.5,-15.5);

\draw[thick] (-1.5,-12.5) -- (-11,-15.5);
\draw[thick] (-0.5,-12.5) -- (-4,-15.5);
\draw[thick] (0.5,-12.5) -- (2,-15.5);
\draw[thick] (1.5,-12.5) -- (10,-15.5);

\draw[thick] (6.5,-12.5) -- (3.5,-15.5);
\draw[thick] (8.5,-12.5) -- (17,-15.5);
\draw[thick] (10.5,-12.5) -- (25,-15.5);

\draw[thick] (17,-12.5) -- (12,-15.5);
\draw[thick] (17.7,-12.8) -- (27,-15.5);
\draw[thick] (20.1,-12.7) -- (35,-15.5);

\node at (-39,-17) {$($};\diag{-38.5}{-16.5}{3};\node at (-34,-17) {$,\,\emptyset)$};
\node at (-30,-17) {$($};\diagg{-29.5}{-16.5}{2}{1};\node at (-26,-17) {$,\,\emptyset)$};
\node at (-22,-17) {$($};\diaggg{-21.5}{-16.5}{1}{1}{1};\node at (-19,-17) {$,\,\emptyset)$};

\node at (-15,-17) {$($};\diag{-14.5}{-16.5}{2};\node at (-12,-17) {$,$};\diag{-11.5}{-16.5}{1};\node at (-10,-17) {$)$};
\node at (-7,-17) {$($};\diagg{-6.5}{-16.5}{1}{1};\node at (-5,-17) {$,$};\diag{-4.5}{-16.5}{1};\node at (-3,-17) {$)$};

\node at (1,-17) {$($};\diag{1.5}{-16.5}{1};\node at (3,-17) {$,$};\diag{3.5}{-16.5}{2};\node at (6,-17) {$)$};
\node at (9,-17) {$($};\diag{9.5}{-16.5}{1};\node at (11,-17) {$,$};\diagg{11.5}{-16.5}{1}{1};\node at (13,-17) {$)$};

\node at (17,-17) {$(\emptyset\,,$};\diag{18.5}{-16.5}{3};\node at (22,-17) {$)$};
\node at (26,-17) {$(\emptyset\,,$};\diagg{27.5}{-16.5}{2}{1};\node at (30,-17) {$)$};
\node at (34,-17) {$(\emptyset\,,$};\diaggg{35.5}{-16.5}{1}{1}{1};\node at (37,-17) {$)$};
\end{tikzpicture}
\end{center}
Thanks to its properties recalled in Proposition \ref{propE}, the element $E_m^{(x,\alpha_b)}$ in $F\cH{m}$ is a non-zero element proportional to the minimal central idempotent corresponding to the one-dimensional representation associated to $(x,\alpha_b)$. It means that it is non-zero in this one-dimensional representation of $F\cH{m}$ and acts as $0$ in any other irreducible representation. Now if $n\geq m$, it follows that $E_m^{(x,\alpha_b)}$ seen as an element of $\cH{n}$ is non-zero in an irreducible representation if and only if this irreducible representation contains in its restriction to $\cH{m}$ the given one-dimensional representation. Therefore the quotient by $E_m^{(x,\alpha_b)}=0$ removes exactly these irreducible representations.

In the particular case of $F\cA_n$, which is the quotient of $\cH{n}$ by the relation $E_2^{(-q^{-1},\alpha_2)}=0$, we recall the indexing of one-dimensional representations set up in (\ref{1dimrep}), and we find that the disappearing representations are those $V_{(\lambda,\mu)}$ with $\mu$ having at least two non-empty rows. We summarise this discussion in the following proposition.
\begin{prop}\label{prop:repA}
The algebra $F\cA_n$ is semisimple with the following set of irreducible representations:
\[Irr(F\cA_n)=\{V_{(\lambda,\mu)}\ |\ (\lambda,\mu)\in Par_2(n)\ \text{and}\ \ell(\mu)<2\}\ .\]
\end{prop}
The Bratteli diagram for the algebras $\cA_n$ is obtained from the Bratteli diagram for the algebras $\cH{n}$ above, where all bipartitions with more than one row in the second component are removed:
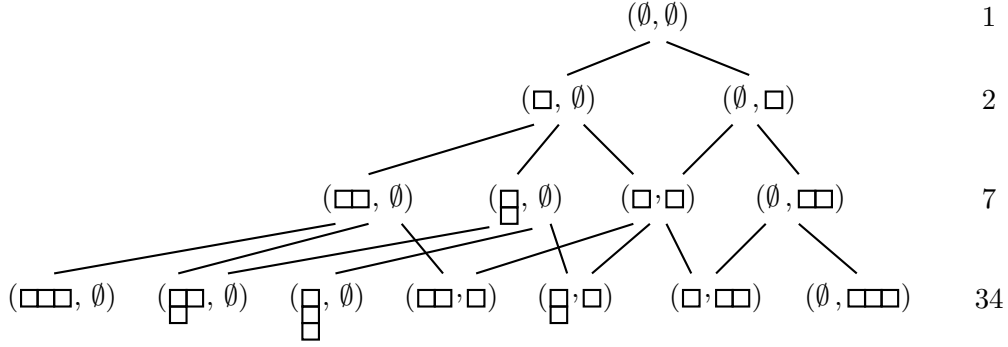
\begin{center}
\begin{tikzpicture}[scale=0.22]
\node at (0,0) {$(\emptyset,\emptyset)$};
\node at (20,0) {$1$};

\draw[thick] (-0.5,-1.5) -- (-5.5,-3.5);
\draw[thick] (0.5,-1.5) -- (5.5,-3.5);

\node at (-8,-5) {$($};\diag{-7.5}{-4.5}{1};\node at (-5,-5) {$,\,\emptyset)$};
\node at (5,-5) {$(\emptyset\,,$};\diag{6.5}{-4.5}{1};\node at (8,-5) {$)$};
\node at (20,-5) {$2$};

\draw[thick] (-7.5,-6.5) -- (-17.5,-9.5);
\draw[thick] (-6,-6.5) -- (-8.5,-9.5);
\draw[thick] (-4.5,-6.5) -- (-1.5,-9.5);
\draw[thick] (4.5,-6.5) -- (1.5,-9.5);
\draw[thick] (6,-6.5) -- (8.5,-9.5);

\node at (-20,-11) {$($};\diag{-19.5}{-10.5}{2};\node at (-16,-11) {$,\,\emptyset)$};
\node at (-10,-11) {$($};\diagg{-9.5}{-10.5}{1}{1};\node at (-7,-11) {$,\,\emptyset)$};
\node at (-2,-11) {$($};\diag{-1.5}{-10.5}{1};\node at (0,-11) {$,$};\diag{0.5}{-10.5}{1};\node at (2,-11) {$)$};
\node at (7,-11) {$(\emptyset\,,$};\diag{8.5}{-10.5}{2};\node at (11,-11) {$)$};
\node at (20,-11) {$7$};

\draw[thick] (-19.5,-12.5) -- (-36.5,-15.5);
\draw[thick] (-17.5,-12.5) -- (-29,-15.5);
\draw[thick] (-15.5,-12.5) -- (-13,-15.5);

\draw[thick] (-10.2,-12.7) -- (-26,-15.5);
\draw[thick] (-7.6,-12.8) -- (-19.5,-15.5);
\draw[thick] (-6.5,-12.5) -- (-5.5,-15.5);

\draw[thick] (-1.5,-12.5) -- (-11,-15.5);
\draw[thick] (-0.5,-12.5) -- (-4,-15.5);
\draw[thick] (0.5,-12.5) -- (2,-15.5);

\draw[thick] (6.5,-12.5) -- (3.5,-15.5);
\draw[thick] (8.5,-12.5) -- (12,-15.5);

\node at (-39,-17) {$($};\diag{-38.5}{-16.5}{3};\node at (-34,-17) {$,\,\emptyset)$};
\node at (-30,-17) {$($};\diagg{-29.5}{-16.5}{2}{1};\node at (-26,-17) {$,\,\emptyset)$};
\node at (-22,-17) {$($};\diaggg{-21.5}{-16.5}{1}{1}{1};\node at (-19,-17) {$,\,\emptyset)$};
\node at (-15,-17) {$($};\diag{-14.5}{-16.5}{2};\node at (-12,-17) {$,$};\diag{-11.5}{-16.5}{1};\node at (-10,-17) {$)$};
\node at (-7,-17) {$($};\diagg{-6.5}{-16.5}{1}{1};\node at (-5,-17) {$,$};\diag{-4.5}{-16.5}{1};\node at (-3,-17) {$)$};
\node at (1,-17) {$($};\diag{1.5}{-16.5}{1};\node at (3,-17) {$,$};\diag{3.5}{-16.5}{2};\node at (6,-17) {$)$};
\node at (10,-17) {$(\emptyset\,,$};\diag{11.5}{-16.5}{3};\node at (15,-17) {$)$};
\node at (20,-17) {$34$};
\end{tikzpicture}
\end{center}
The dimension of $F\cA_n$ can be easily calculated, by summing the squares of the dimensions of the irreducible representations:
\[\dim(F\cA_n)=\sum_{i=0}^n\sum_{\lambda\vdash n-i}(\dim V_{(\lambda,(i))})^2=\sum_{i=0}^n\binom{n}{i}^2\sum_{\lambda\vdash n-i}d_{\lambda}^2=\sum_{i=0}^n\binom{n}{i}^2(n-i)!\ ,\]
where we first split the sum according to the size of the second partition $\mu$, which must be a single line of $i$ boxes, and then we use successively \eqref{dimbipart} and \eqref{dimfactn}. The dimensions for $n=0,1,2,3$ are written in the diagram above.

\begin{rem}\label{rem-semi}
The above description of the representations of $F\cA_n$ is also valid for many specialisations of the parameters $\alpha_1,\alpha_2,q$ in $\cA_n$, namely, those specialisations such that the algebra $\cH{n}$ is semisimple. This happens if $q^2$ is not a root of unity of order $e\leq n$ and $\alpha_1\neq \alpha_2q^{\pm2i}$ for $i=1,\dots,n-1$, see \cite{Ari}.
\end{rem}

\paragraph{An $R$-basis of $\cA_n$.} A word $b=b_1b_2\dots b_n$ of $B_n$ is said to be $\bar{1}\bar{2}$-avoiding if all barred numbers in $b$ appear in decreasing order, see \cite{Sim}. Put differently, $b$ avoids the pattern $\bar{1}\bar{2}$ if there are no two indices $1\leq i<j \leq n$ such that $b_{i}=\overline{m}_1$ and $b_{j}=\overline{m}_2$ with $0<m_{1}<m_{2}$. For instance, $35\bar{6}1\bar{4}\bar{2}$ is $\bar{1}\bar{2}$-avoiding in $B_{6}$ while $35\bar{4}1\bar{6}\bar{2}$ is not because of the subsequence $\bar{4}\bar{6}$. 

We will denote by $\avB$ the subset of all signed permutations in $B_n$ which are $\bar{1}\bar{2}$-avoiding. A word $b=b_1b_2\dots b_n$ corresponding to a permutation in $\avB$ can be written as follows: choose $i$ numbers in $\{1,2,\dots,n\}$ that will be barred, choose $i$ positions among $n$ to place these barred numbers in decreasing order in $b$, and then permute the remaining $n-i$ numbers in the remaining $n-i$ positions in $b$. It follows that:
\begin{equation}
	|\avB| = \sum_{i=0}^n (n-i)! \binom{n}{i}^2\ . \label{eq:cardavoid}
\end{equation} 

We are now ready to give a basis of $\cA_n$ over $R$. Thus, this also gives a basis for any specialisation of $q,\alpha_1,\alpha_2$ to non-zero complex numbers.
\begin{thm}\label{theo:basisA}
	The algebra $\cA_n$ is free over $R$ with basis given by the set of elements $g_\omega$ corresponding to $\bar{1}\bar{2}$-avoiding signed permutations, that is, the basis is:
	\begin{equation}
		\{ g_\omega \ | \ \omega \in \avB \}. \label{eq:basisA}
	\end{equation}
\end{thm}
\begin{proof}
We first prove that the set \eqref{eq:basisA} is a spanning set:
\begin{equation}
		\cA_n=\rmspan_R\{ g_\omega \ | \ \omega \in \avB \}. \label{eq:spanA}
	\end{equation}
	Since $\cA_n$ is a quotient of $\cH{n}$, it is clearly linearly generated by the set of elements $g_\omega$ with $\omega \in B_n$. To show that this spanning set can be reduced to \eqref{eq:spanA}, we proceed by induction on the length of elements in $B_n$.
	
	Any signed permutation $\omega \in B_n$ of length $\ell(\omega) < 4$ is such that at most one number in $\{1,2,\dots,n\}$ is mapped to a barred number. Therefore, all signed permutations $\omega$ with $\ell(\omega) < 4$ are $\bar{1}\bar{2}$-avoiding and the associated elements $g_\omega$ belong to the spanning set \eqref{eq:spanA}.
	
Suppose now that all elements $g_\omega$ with $\ell(\omega)\leq m$, for some fixed integer $m \geq 4$, belong to the span of the set \eqref{eq:basisA}. Consider an element $g_\omega$ with $\ell(\omega)=m+1$ such that $\omega\notin \avB$. It follows by \cite[Lemma 2.1]{Stem} that $\omega$ must contain $s_0s_1s_0s_1$ in some reduced expression. This implies in turn that there is a reduced expression for $g_\omega$ that contains $g_0g_1g_0g_1$. Relation \eqref{eq:relquotient} can hence be used to express $g_\omega$ as a linear combination of terms of length less than $m+1$. By induction hypothesis, we therefore conclude that $g_\omega$ can be written in terms of the set \eqref{eq:basisA}.

Now, since the set (\ref{eq:basisA}) is a spanning set of $\cA_n$ over $R$,  it is also a spanning set of $F\cA_n$ over $F$. Moreover, the dimension of $F\cA_n$ over $F$ was calculated after Proposition \ref{prop:repA} and it coincides with the cardinality of the spanning set. Therefore, the set (\ref{eq:basisA}) is a basis of $F\cA_n$ and in particular is linearly independent over $F$. Thus it is also linearly independent over $R$ (since $R\subset F$). We conclude that the set (\ref{eq:basisA}) is an $R$-basis of $\cA_n$.
\end{proof}

\begin{rem}
The elements of $B_n$ which avoid the pattern $\bar{1}\bar{2}$ are in fact the same as the elements which do not contain $s_0s_1s_0s_1$ in any reduced expression, see \cite[Lemma 2.1]{Stem}.
\end{rem}

\subsection{Quasi-idempotents in $\cA_n$}\label{subsec-idemA}

Let $n\geq 2$. In $\cA_n$, the element $E_n^{(-q^{-1},\alpha_2)}$, which is proportional to $E_2^{(-q^{-1},\alpha_2)}$ is equal to $0$. Thus among the four central quasi-idempotents $E_n^{(x,\alpha_b)}$ of $\cH{n}$, only three remain non-zero in $\cA_n$:
\[E_n^{(-q^{-1},\alpha_1)}\ ,\ \ \ E_n^{(q,\alpha_2)}\,,\ \ \ E_n^{(q,\alpha_1)}\ .\]
In $\cH{n}$, no common factor appears in the coefficients of these elements. This will now be different in $\cA_n$, where a non-trivial common factor in the ring $R$ may sometimes be factored out. Thus we can define renormalised elements defined over $R$ by removing this common factor, and it will be important for later use to do so. In fact, only the two first elements in the list above factorise generically over the ring $R$, and this is the content of the following statement.
\begin{prop}\label{prop-idempAn}
Let $n\geq 2$. We have in $\cA_n$: 
\begin{align} 
& E_n^{(-q^{-1},\alpha_1)}=\alpha_2^{-1}\prod_{i=0}^{n-2}\left(1-\frac{\alpha_1}{\alpha_2}q^{-2i}\right)\tilde{E}_n^{(-q^{-1},\alpha_1)}\,,\label{eq:idemENrenorm1}\\
&E_{n}^{(q,\alpha_2)}=q^{\frac{n(n-1)}{2}}[n]_q!\tilde{E}_{n}^{(q,\alpha_2)}\,,\label{eq:idemENrenorm2}
\end{align}
where the renormalised elements are given by:
\begin{align} 
& \tilde{E}_n^{(-q^{-1},\alpha_1)}=\Lambda_n(g_1,\dots,g_{n-1})\cdot \bigl(\alpha_2+\alpha_1q^{-(n-2)}[n-1]_q-\sum_{i=0}^{n-1}(-1)^{i}q^{-i}g_0\dots g_i\bigr)\,,\label{eq:idemENtilde1}\\
& \tilde{E}_{n}^{(q,\alpha_2)}=\tilde{E}_{n-1}^{(q,\alpha_2)}\bigl((1-q^2)(1+q g_{n-1}+\dots +q^{n-2}g_{n-1}\dots g_2)+q^{n-1}g_{n-1}\dots g_1(1-\alpha_1^{-1}g_0)\bigr)\ ,\label{eq:idemENtilde2}
\end{align}
with the convention that $\tilde{E}_{1}^{(q,\alpha_2)}=E_1^{(q,\alpha_2)}=(1-\alpha_1^{-1}g_0)$.
\end{prop}
\begin{proof}
\textbf{1.} First we prove \eqref{eq:idemENrenorm1}. Denote $x=-q^{-1}$ and
\[X_n^0=(1-x^{n-1}\alpha_2^{-1}g_0g_1\dots g_{n-1})(1-x^{n-2}\alpha_2^{-1}g_0g_1\dots g_{n-2})\ldots\ldots (1-\alpha_2^{-1}g_0)\ .\]
Recall from \eqref{eq:ELambda} that we have
\[
E_n^{(x,\alpha_1)}=\Lambda_n\cdot X_n^0\,, \ \ \ \ \text{where $\Lambda_n=\Lambda^x_n(g_1,\dots,g_{n-1})$.}
\]
We use induction on $n$. For $n=2$, the formula is easy to check by direct calculation. We write the definition of $E_2^{(x,\alpha_1)}$ and use the defining relation of $\cA_n$ to replace $g_0g_1g_0g_1$. Now, let $n\geq 3$ and write the above formula as
\begin{equation}\label{form-E-Xn0}
E_n^{(x,\alpha_1)}=\Lambda_n\cdot (1-x^{n-1}\alpha_2^{-1}g_0g_1\dots g_{n-1})X_{n-1}^0\,.
\end{equation}
Note that
\[\Lambda_n=A\cdot \Lambda^x_{n-1}(g_1,\dots,g_{n-2})=B\cdot \Lambda^x_{n-1}(g_2,\dots,g_{n-1})\ ,\ \ \ \text{for some $A,B\in\cH{n}$.}\]
Besides, $\Lambda^x_{n-1}(g_2,\dots,g_{n-1})g_0g_1\dots g_{n-1}=g_0g_1\dots g_{n-1}\Lambda^x_{n-1}(g_1,\dots,g_{n-2})$ using the braid relations. Therefore, in \eqref{form-E-Xn0}, we can bring $\Lambda^x_{n-1}(g_1,\dots,g_{n-2})$ in front of $X_{n-1}^0$ and thus use the induction hypothesis. At this point, we have:
\begin{equation}\label{form-E-Xn0bis}
E_n^{(x,\alpha_1)}=\gamma_{n-1}\Lambda_n\cdot (1-x^{n-1}\alpha_2^{-1}g_0g_1\dots g_{n-1})\bigl(\alpha_2+\alpha_1q^{-(n-3)}[n-2]_q-\sum_{i=0}^{n-2}x^ig_0\dots g_i\bigr)\,,
\end{equation}
with the coefficient $\gamma_{n-1}$ given by the induction hypothesis. Now we are going to use the defining relation of $\cA_n$ in the following form
\[(1-q^{-1}g_1)g_0g_1g_0=(1-q^{-1}g_1)(q\alpha_1^2-q\alpha_1g_0+\alpha_1g_0g_1)\ ,\]
and the fact that $\Lambda_n=C\cdot (1-q^{-1}g_1)$ for some $C\in \cH{n}$ to make the following calculation, recalling that $\Lambda_ng_k=x\Lambda_n$ for $k=1,\dots,n-1$,
\[\begin{array}{rl} & \displaystyle\Lambda_n\cdot g_0g_1\dots g_{n-1}\sum_{i=0}^{n-2}x^ig_0g_1\dots g_i\\
= & \displaystyle\Lambda_n\cdot (q\alpha_1^2-q\alpha_1g_0+\alpha_1g_0g_1)g_2\dots g_{n-1}\sum_{i=0}^{n-2}x^ig_1\dots g_i \\
= & \displaystyle\Lambda_n\cdot (x^{n-2}q\alpha_1^2-x^{n-2}q\alpha_1g_0+\alpha_1g_0g_1\dots g_{n-1})\sum_{i=0}^{n-2}x^ig_1\dots g_i \\[1.2em]
= & \displaystyle\Lambda_n\cdot \Bigl(q^{-(n-3)}x^{n-2}\alpha_1^2[n-1]_q+q^{-(n-2)}\alpha_1[n-1]_qg_0g_1\dots g_{n-1}\\
& \displaystyle \hspace{8cm} -x^{n-2}q\alpha_1 \sum_{i=0}^{n-2}x^ig_0g_1\dots g_i\Bigr)\ .
\end{array}\]
We have used the braid relations to move $g_1\dots g_i$ through $g_0g_1\dots g_{n-1}$ (getting $g_2\dots g_{i+1}$), and that $\sum_{i=0}^{n-2}x^ig_1\dots g_i=q^{-(n-2)}[n-1]_q$ when all $g$'s are replaced by $x=-q^{-1}$.

It remains only to use this formula in \eqref{form-E-Xn0bis} and to collect the various terms. Omitting $\gamma_{n-1}\Lambda_n$ which is in front of everyone, one finds directly that the coefficient in front of
$g_0g_1\dots g_i$ when $i<n-1$ is $-x^i(1-q^{-2(n-2)}\frac{\alpha_1}{\alpha_2})$. Then easy manipulations give that the coefficients in front of $1$ and in front of $g_0g_1\dots g_{n-1}$ are respectively,
\[\left(1-q^{-2(n-2)}\frac{\alpha_1}{\alpha_2}\right)(\alpha_2+\alpha_1q^{-(n-2)}[n-1]_q)\ \ \ \ \text{and}\ \ \ \ -\left(1-q^{-2(n-2)}\frac{\alpha_1}{\alpha_2}\right)x^{n-1}\ .\]
This concludes the verification of \eqref{eq:idemENrenorm1}.

\vskip .2cm
\textbf{2.} Now we prove \eqref{eq:idemENrenorm2} with similar methods, using induction on $n$. Once again, the case $n=2$ is directly verified using the explicit expression for $E_{2}^{(q,\alpha_2)}$ and the defining relation of $\cA_n$ to replace $g_0g_1g_0g_1$. For $n\geq 3$, the recursive formula \eqref{eq:Erec} allows to write
\begin{equation}
	E^{(q,\alpha_2)}_n = E^{(q,\alpha_2)}_{n-1}\Big(1+\sum_{i=1}^{n-1}q^{n-i}g_{n-1}\dots g_{i}-q^{n-1}\alpha_{1}^{-1}g_{n-1}\dots g_1g_0(1+\sum_{i=1}^{n-1}q^ig_1\dots g_i) \Big) \ . \label{eq:Eq2rec} 
\end{equation}  
Then, we use in \eqref{eq:Eq2rec} the defining relation of $\cA_n$, which can be rewritten as
\[(-\alpha_1^{-1})(1-\alpha_1^{-1}g_0)g_1g_0g_1=q(1-\alpha_1^{-1}g_0)(-q+g_1-\alpha_1^{-1}g_1g_0)\ ,\]
together with the fact that $E^{(q,\alpha_2)}_{n-1}=C\cdot (1-\alpha_1^{-1}g_0)$ for some $C \in \cH{n}$ to get
\begin{align}
E^{(q,\alpha_2)}_n = E^{(q,\alpha_2)}_{n-1}\Big(1&+\sum_{i=2}^{n-1}q^{n-i}g_{n-1}\dots g_{i}+q^{n-1}g_{n-1}\dots g_{1}(1-\alpha_1^{-1}g_0) \nonumber \\
&+q^{n+1}g_{n-1}\dots g_2 (-q+g_1-\alpha_1^{-1}g_1g_0)(1+\sum_{i=2}^{n-1}q^{i-1}g_2\dots g_i) \Big) \ . \label{eq:Eq2rec2}
\end{align}
Recalling that $E^{(q,\alpha_2)}_{n-1}g_k = q E^{(q,\alpha_2)}_{n-1}$ for $k=1,\dots, n-2$, we can use the braid relations and the Hecke relation to obtain:
\begin{align*}
&E^{(q,\alpha_2)}_{n-1} g_{n-1}\dots g_2(1+\sum_{i=2}^{n-1}q^{i-1}g_2\dots g_i) = E^{(q,\alpha_2)}_{n-1} q^{n-2} ( 1 +\sum_{i=2}^{n-1}q^{n-i}g_{n-1}\dots g_{i}) \ , \\
&E^{(q,\alpha_2)}_{n-1}g_{n-1}\dots g_2 g_1(1-\alpha_1^{-1}g_0)(1+\sum_{i=2}^{n-1}q^{i-1}g_2\dots g_i) = E^{(q,\alpha_2)}_{n-1}q^{n-2}[n-1]_qg_{n-1}\dots g_2 g_1(1-\alpha_1^{-1}g_0)  \ . 
\end{align*}
Replacing these results in \eqref{eq:Eq2rec2} and combining terms together, it is found that
\begin{align*}
E^{(q,\alpha_2)}_n = E^{(q,\alpha_2)}_{n-1}q^{n-1}[n]_q\Big((1-q^2)(1+\sum_{i=2}^{n-1}q^{n-i}g_{n-1}\dots g_{i})+q^{n-1}g_{n-1}\dots g_{1}(1-\alpha_1^{-1}g_0)\Big)  \ .
\end{align*}
The proof is completed by using the induction hypothesis on $E^{(q,\alpha_2)}_{n-1}$.
\end{proof}

\section{The one-boundary Temperley--Lieb algebra and its $gl_N$-generalisations}\label{sec:C}

Let $N\geq 2$. In this section we define the symmetric one-boundary $N$-centraliser algebras $\cC_{n,N}$ as quotients of the algebra $\cA_n$. The meaning of this definition will be clear from the point of view of representation theory, and will result in a natural description of the semisimple representation theory of $\cC_{n,N}$. Besides, our motivation and the origin of the terminology comes from the use we will make of the algebras $\cC_{n,N}$ to describe $U_q(gl_N)$-centralisers in Section \ref{sec-cent}. 

We will then study in details the case $N=2$, showing that we recover the generic 3-parameter one-boundary Temperley--Lieb algebra, for which we will describe a basis using the signed permutations from the preceding section.

\subsection{Definition}

Recall that we have defined in \eqref{eq:symHecke} the $q$-antisymmetriser $\Lambda^{-q^{-1}}_{N+1}(g_1,\dots,g_N)$ of the usual Hecke algebra generated by $g_1,\dots,g_N$. Recall also that we have from Section \ref{subsec-idemA} the following element:
\begin{equation} 
\tilde{E}_n^{(-q^{-1},\alpha_1)}=\Lambda_n(g_1,\dots,g_{n-1})\cdot \bigl(\alpha_2+\alpha_1q^{-(n-2)}[n-1]_q-\sum_{i=0}^{n-1}(-1)^{i}q^{-i}g_0\dots g_i\bigr)\,,\label{eq:idemENtilde1_bis}
\end{equation}
which is a renormalisation in $\cA_n$ of the quasi-idempotent $\tilde{E}_N^{(-q^{-1},\alpha_1)}$. We propose the following definition.
\begin{defi}\label{defCnN}
We define the symmetric one-boundary $N$-centraliser algebra $\cC_{n,N}$ to be the quotient of the algebra $\cA_n$ by the relations:
\begin{align}
&\tilde{E}_N^{(-q^{-1},\alpha_1)}=0\,, \label{eq:relquotient2N}\\
&\Lambda^{-q^{-1}}_{N+1}(g_1,\dots,g_N)=0\,. \label{eq:relquotient3N} 
\end{align}
It is understood that $\cC_{n,N}=\cA_n$ when $n<N$.
\end{defi}

\paragraph{Semisimple representation theory.} In this paragraph, we extend the algebra $\cC_{n,N}$ over the field of fractions $F$ and denote it $F\cC_{n,N}$. The description below is also valid for specialisations of the parameters satisfying the semisimplicity conditions for $\cH{n}$ in Remark \ref{rem-semi}.

\begin{prop}\label{prop:repCN}
The algebra $F\cC_{n,N}$ is semisimple with the following set of irreducible representations:
\[Irr(F\cC_{n,N})=\{V_{(\lambda,\mu)}\ |\ (\lambda,\mu)\in Par_2(n)\ \text{with}\ \ell(\lambda)\leq N-1\ \text{and}\ \ell(\mu)\leq 1\}\ .\]
\end{prop}
\begin{proof}
In the algebra $\cC_{n,N}$, we have:
\[E_N^{(-q^{-1},\alpha_1)}=0\ \ \ \ \ \text{and}\ \ \ \ \ E_2^{(-q^{-1},\alpha_2)}=0\ .\]
Reproducing the same reasoning as in the preceding section before Proposition \ref{prop:repA}, we find that cancelling these two elements kills the irreducible representation $V_{(\lambda,\mu)}$ of $F\cH{n}$ if and only if either $\lambda$ has strictly more than $N-1$ rows or $\mu$ has strictly more than one row.

It remains to argue that, in the remaining irreducible representations the last relation \eqref{eq:relquotient3N} cancelling the $q$-antisymmetriser in $N$ generators $g_1,\dots,g_{N}$ is satisfied. This can be checked rather directly, using the explicit description of the irreducible representations of $\cH{n}$ \cite{AK,Ho,OP11} and using the same sort of methods than those used in \cite{CP} for the $q$-symmetriser.

Otherwise, note that the claim is equivalent to the fact that the one-dimensional representation of $FH_{N+1}$ given by $g_1,\dots g_{N}\mapsto -q^{-1}$ (that is, indexed by a one-column partition) does not appear when we restrict to $FH_{N+1}$ the irreducible representations of $F\cH{N+1}$ indexed by bipartitions $(\lambda,\mu)$ with $\ell(\lambda)\leq N-1$ and $\ell(\mu)\leq 1$. These restrictions are expressed in terms of Littlewood--Richardson coefficients (see for example \cite{ChP,Ram}). This implies in particular the easy claim above since, by what is called the Pieri rule, the Littlewood--Richardson coefficient $d_{\lambda,\mu}^{\nu}$ is $0$ with $\lambda$ and $\mu$ as above and $\nu$ the one-column partition of length $N+1$.
\end{proof}

The Bratteli diagram for the algebras $F\cC_{n,N}$ is thus obtained from the Bratteli diagram for the algebras $F\cA_n$ given before, removing all bipartitions with $N$ rows or more in the first component. The dimension of $F\cC_{n,N}$ is then calculated as follows:
\[\dim(F\cC_{n,N})=\sum_{i=0}^n\binom{n}{i}^2\sum_{\substack{\lambda\vdash i \\\ell(\lambda)<N}}d_{\lambda}^2\ .\]
For $N=2$, the algebra $F\cC_{n,2}$ will be studied in details in the next subsection.

For $N=3$, the sum of the squares of $d_{\lambda}$'s in the formula above is the Catalan number $\frac{1}{i+1}\binom{2i}{i}$. Moreover, the series of dimensions start with $1,2,7,33,183$ and is the series labelled A086618 in \cite{OEIS}. We note that, similarly to the situation $N=2$ discussed in Remark \ref{rc321avoiding}, the dimension of $F\cC_{n,3}$ is the number of signed permutations of $\{-n,\dots,-1,1,\dots,n\}$ which are 4321-avoiding \cite{Eg}.

In general, we may ask whether the algebra $\cC_{n,N}$ is free over the ring $R$ and we may look for a basis indexed by a natural subset of signed permutations, for example those avoiding certain patterns. This is what we are going to do in details for $N=2$ in the next subsection.

\begin{rem}
The representation theory shows that over $F$ the relation \eqref{eq:relquotient3N} is actually implied by the others. This is not true over $R$ but this is also true in any specialisation such that the algebra $\cH{n}$ is semisimple (Remark \ref{rem-semi}).
\end{rem}

\subsection{The case $N=2$ (the Temperley--Lieb situation)}

The definition of $\cC_{n,N}$ can be written slightly differently and more explicitly when $N=2$. 
\begin{prop}\label{propCn2}
The algebra $\cC_{n,2}$ is the quotient of the algebra $\cH{n}$ by the relations:
\begin{align}
&g_1g_0g_1 = q(\alpha_1+\alpha_2)(q-g_1)-q^2g_0 +q(g_0g_1+g_1g_0)\,, \label{eq:relquotient2}\\
&g_ig_{i+1}g_i = q^3-q^2(g_i+g_{i+1})+q(g_ig_{i+1}+g_{i+1}g_i)\ \ \ \ \ \ i=1,\dots,n-2\,. \label{eq:relquotient3} 
\end{align}
\end{prop}
\begin{proof}
The first relation is obtained easily by writing explicitly $\tilde{E}_2^{(-q^{-1},\alpha_1)}=0$ using Proposition \ref{prop-idempAn}. The second relation for $i=1$ is $\Lambda^{-q^{-1}}_3(g_1,g_2)=0$. It implies, by suitable conjugation, the second relations for any $i\geq 1$. The last statement is that \eqref{eq:relquotient2} implies the defining relation of $\cA_n$:
\[g_0g_1g_0g_1=-q^{2}\alpha_1^{2}+q\alpha_1^{2}g_1+q^{2}\alpha_1g_0-q\alpha_1(g_1g_0+g_0g_1) +\alpha_1 g_1g_0g_1+qg_0g_1g_0\ .\]
This is easy to see since multiplying \eqref{eq:relquotient2} on the left or on the right by $g_0$ gives 
\begin{equation}
	g_0g_1g_0g_1 = g_1g_0g_1g_0 = q\alpha_1\alpha_2(q-g_1)+q g_0g_1g_0\ , \label{eq:relquotient32}
\end{equation}
which, combined with \eqref{eq:relquotient2}, produces the desired relation.
\end{proof}

\paragraph{A Temperley--Lieb presentation.} We make the slight change of generators as follows:
\begin{equation}
	e_0:=\alpha_2-g_0, \quad e_i:=q-g_i, \quad i=1,2,\dots,n-1.
\end{equation}
Then it is an easy exercice to check that the algebra $\cC_{n,2}$ can be equivalently presented as generated by $e_i$ for $i=0,1,\dots,n-1$ with the following defining relations:
\begin{alignat}{2}
	&e_i^2=(q+q^{-1})e_i, && \quad 1\leq i \leq n-1, \label{eq:TLrel1}\\
	&e_0^2=(\alpha_2-\alpha_1)e_0, && \label{eq:TLrel2} \\
	&e_ie_j=e_je_i, && \quad |i-j|\geq 2, \label{eq:TLrel3}\\
	&e_ie_{i\pm 1}e_i=e_i, &&\quad 1\leq i,i\pm 1 \leq n-1, \label{eq:TLrel4} \\	
	&e_1e_0e_1 = (q^{-1}\alpha_2-q\alpha_1)e_1. && \label{eq:TLrel5}	
\end{alignat}
We recover a three-parameter version of the one-boundary Temperley--Lieb algebra, or blob algebra, see for example \cite{tD,LS,MS,MRR,NRdG}.

\paragraph{Semisimple representation theory.}  We take $N=2$ in the Proposition \ref{prop:repCN}.
\begin{prop}\label{prop:repC2}
The algebra $F\cC_{n,2}$ is semisimple with the following set of irreducible representations:
\[Irr(F\cC_{n,2})=\{V_{(\lambda,\mu)}\ |\ (\lambda,\mu)\in Par_2(n)\ \text{and}\ \ell(\lambda),\ell(\mu)\leq 1\}\ .\]
\end{prop}
The Bratteli diagram for the algebras $F\cC_{n,2}$ is obtained from the Bratteli diagram for the algebras $\cH{n}$, where all bipartitions with more than one row in any component are removed. One finds the Pascal triangle, whose beginning is given here:
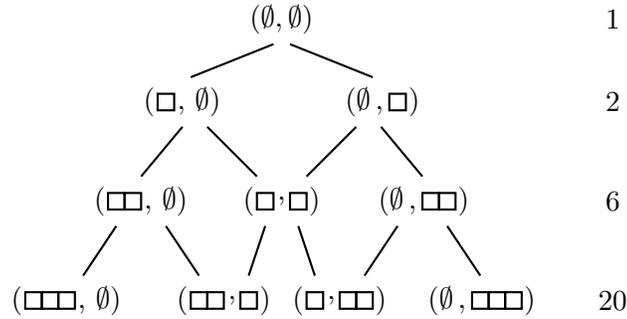
\begin{center}
\begin{tikzpicture}[scale=0.22]
\node at (0,0) {$(\emptyset,\emptyset)$};
\node at (20,0) {$1$};

\draw[thick] (-0.5,-1.5) -- (-5.5,-3.5);
\draw[thick] (0.5,-1.5) -- (5.5,-3.5);

\node at (-8,-5) {$($};\diag{-7.5}{-4.5}{1};\node at (-5,-5) {$,\,\emptyset)$};
\node at (5,-5) {$(\emptyset\,,$};\diag{6.5}{-4.5}{1};\node at (8,-5) {$)$};
\node at (20,-5) {$2$};

\draw[thick] (-6,-6.5) -- (-8.5,-9.5);
\draw[thick] (-4.5,-6.5) -- (-1.5,-9.5);
\draw[thick] (4.5,-6.5) -- (1.5,-9.5);
\draw[thick] (6,-6.5) -- (8.5,-9.5);

\node at (-11,-11) {$($};\diag{-10.5}{-10.5}{2};\node at (-7,-11) {$,\,\emptyset)$};
\node at (-2,-11) {$($};\diag{-1.5}{-10.5}{1};\node at (0,-11) {$,$};\diag{0.5}{-10.5}{1};\node at (2,-11) {$)$};
\node at (7,-11) {$(\emptyset\,,$};\diag{8.5}{-10.5}{2};\node at (11,-11) {$)$};
\node at (20,-11) {$6$};

\draw[thick] (-10,-12.5) -- (-12,-15.5);
\draw[thick] (-7,-12.5) -- (-5,-15.5);

\draw[thick] (-1,-12.5) -- (-2,-15.5);
\draw[thick] (1,-12.5) -- (2,-15.5);

\draw[thick] (7,-12.5) -- (5,-15.5);
\draw[thick] (10,-12.5) -- (12,-15.5);

\node at (-16,-17) {$($};\diag{-15.5}{-16.5}{3};\node at (-11,-17) {$,\,\emptyset)$};
\node at (-6,-17) {$($};\diag{-5.5}{-16.5}{2};\node at (-3,-17) {$,$};\diag{-2.5}{-16.5}{1};\node at (-1,-17) {$)$};
\node at (1,-17) {$($};\diag{1.5}{-16.5}{1};\node at (3,-17) {$,$};\diag{3.5}{-16.5}{2};\node at (6,-17) {$)$};
\node at (10,-17) {$(\emptyset\,,$};\diag{11.5}{-16.5}{3};\node at (15,-17) {$)$};
\node at (20,-17) {$20$};
\end{tikzpicture}
\end{center}
The dimension of $F\cC_{n,2}$ is then easily calculated:
\begin{equation}
\dim(F\cC_{n,2})=\sum_{i=0}^n\binom{n}{i}^2=\binom{2n}{n}\ . \label{eq:dimFCn2}
\end{equation}

\paragraph{An $R$-basis of $\cC_{n,2}$.} Now we produce an $R$-basis of $\cC_{n,2}$ in terms of signed permutations, and also explicitly in terms of the generators.

We consider the signed permutations with the following avoiding patterns: $(\pm1,-2)$ and $(\pm3,2,\pm1)$. This means that in the word $b_1b_2\dots b_n$ giving the signed permutations ($b_i\in\{\pm1,\dots,\pm n\}$ is the image of $i$), we never have:
\begin{itemize}
\item For $i<j$: $b_j<0$ and $|b_i|<|b_j|$ (in words, a negative number $b_j$ is never preceded by a smaller number when ignoring signs);
\item For $i<j<k$: $b_j>0$ and $|b_i|>b_j>|b_k|$ (in words, a positive number is never in the middle of a decreasing sequence of length 3, when ignoring signs).
\end{itemize}
We denote by $B_n(\bar{1}\bar{2},1\bar{2},321,\bar{3}21,32\bar{1},\bar{3}2\bar{1})$ the set of signed permutations with these avoiding patterns. These elements are called  fully commutative top elements in \cite{Stem}. For example, $3\bar{2}45\bar{1}$ is in this set for $n=5$, while $3\bar{2}\bar{5}41$ is not for three reasons: the subsequences $3\bar{5}$, $\bar{2}\bar{5}$ and $\bar{5}41$.

In terms of the standard basis elements \eqref{eq:basecH}, it is proved in  \cite[Corollary 5.6]{Stem} that the set $g_\omega$ with $\omega\in B_n(\bar{1}\bar{2},1\bar{2},321,\bar{3}21,32\bar{1},\bar{3}2\bar{1})$ corresponds to all elements of the form:
\begin{equation}\label{basis-cC2}
[n_1,m_1][n_2,m_2]\dots [n_r,m_r]\ \ \ \ \ \ \text{with}\ \left\{\begin{array}{l} 0\leq n_1<n_2<\dots<n_r\leq n-1\ \ \text{and}\ \ m_i\leq n_i\ ,\\[0.4em]
0=m_1=\dots=m_s<m_{s+1}<\dots<m_r\ .
\end{array}\right.
\end{equation}
The cardinality of the set $B_n(\bar{1}\bar{2},1\bar{2},321,\bar{3}21,32\bar{1},\bar{3}2\bar{1})$ has been calculated in \cite[Proposition 5.9]{Stem} or \cite[Appendix B]{MRR}, and it is found that:
\begin{equation}\label{eq:cardfullcommtop}
|B_n(\bar{1}\bar{2},1\bar{2},321,\bar{3}21,32\bar{1},\bar{3}2\bar{1})|=\binom{2n}{n}\ .
\end{equation}	
\begin{rem}\label{rc321avoiding}
An alternative description of $B_n(\bar{1}\bar{2},1\bar{2},321,\bar{3}21,32\bar{1},\bar{3}2\bar{1})$ is as follows. Recall that a signed permutation is in particular a permutation of $\{-n,\dots,-1,1,\dots,n\}$. For these permutations on $2n$ elements, there is the usual meaning of being $321$-avoiding (no strictly decreasing subsequence of length $3$ in the sequence of images $b_{-n}\dots b_{-1}b_1\dots b_n$). We leave to the reader to check that the set $B_n(\bar{1}\bar{2},1\bar{2},321,\bar{3}21,32\bar{1},\bar{3}2\bar{1})$ coincides with the set of signed permutations which are $321$-avoiding as permutations of $\{-n,\dots,-1,1,\dots,n\}$. See \cite{Eg} for a proof that this latter set is indeed of cardinal $\binom{2n}{n}$.
\end{rem}
\begin{thm}\label{theo:basisC2}
	The algebra $\cC_{n,2}$ is free over $R$ with basis consisting of elements in (\ref{basis-cC2}), that is, 
	\begin{equation}
		\{ g_\omega \ | \ \omega \in B_n(\bar{1}\bar{2},1\bar{2},321,\bar{3}21,32\bar{1},\bar{3}2\bar{1}) \}\ .\label{eq:basisC2}
	\end{equation}
\end{thm}
One can replace each generator $g_i$ by $e_i$ in the expressions in (\ref{basis-cC2}) and this clearly also gives  a basis of $\cC_{n,2}$. This basis can be found in \cite[Appendix B]{MRR}.
\begin{proof}
We proceed similarly to the proof of Theorem \ref{theo:basisA}. All signed permutations $\omega \in B_n$ with length $\ell(\omega)<3$ avoid the patterns in \eqref{eq:basisC2}. It is shown in \cite{Stem} (see Theorem 4.1 and Corollary 5.6) that an element $\omega$ which does not avoid the patterns of \eqref{eq:basisC2} contains in a reduced expression $s_1s_0s_1$ or $s_is_{i+1}s_i$ for $i\geq 1$. This means that the defining relations \eqref{eq:relquotient2} and \eqref{eq:relquotient3} can be used to express $g_\omega$ in terms of elements of smaller length. Using induction on the length, we can therefore conclude that \eqref{eq:basisC2} is a spanning set for $\cC_{n,2}$ over $R$. The rest follows by comparing the cardinality \eqref{eq:cardfullcommtop} with the dimension \eqref{eq:dimFCn2} of $F\cC_{n,2}$.
\end{proof}

\subsection{Quasi-idempotents in the one-boundary Temperley--Lieb algebra $\cC_{n,2}$}

Let $n\geq 2$. Among the four central quasi-idempotents of $\cH{n}$, the two $E_n^{(-q^{-1},\alpha_i)}$ are equal to $0$ in $\cC_{n,2}$. The one with $\alpha_2$ was cancelled to define $\cA_n$, and the one with $\alpha_1$ was cancelled to define $\cC_{n,2}$. So only the following two remain:
\[E_n^{(q,\alpha_1)}\ ,\ \ \ \ \ \ E_n^{(q,\alpha_2)}\ .\] 
It turns out that the expression of these two remaining quasi-idempotents in $\cC_{n,2}$ simplifies compared to their original definition in $\cH{n}$ and a global factor appears, as shown in the following result. Note that the following result was already true in $\cA_n$ only for $E_n^{(q,\alpha_2)}$ (Proposition \ref{prop-idempAn}). So the novelty here is that it becomes also true for $E_n^{(q,\alpha_1)}$ in $\cC_{n,2}$.
Below, we use again the notation modulo $2$ for the indices of $\alpha_b$.
\begin{prop}\label{prop-idempCn2}
Let $n\geq 2$ and $b\in\{1,2\}$. We have in $\cC_{n,2}$: 
\begin{align} 
&E_{n}^{(q,\alpha_b)}=q^{\frac{n(n-1)}{2}}[n]_q!\tilde{E}_{n}^{(q,\alpha_b)}\,,\label{eq:idemENrenormCn2}
\end{align}
where the renormalised element is given recursively by:
\begin{align} 
& \tilde{E}_{n}^{(q,\alpha_b)}=\tilde{E}_{n-1}^{(q,\alpha_b)}\bigl((1-q^2)(1+q g_{n-1}+\dots +q^{n-2}g_{n-1}\dots g_2)+q^{n-1}g_{n-1}\dots g_1(1-\alpha_{b+1}^{-1}g_0)\bigr)\ ,\label{eq:idemENtildeCn2}
\end{align}
with the convention that $\tilde{E}_{1}^{(q,\alpha_b)}=E_1^{(q,\alpha_b)}=(1-\alpha_{b+1}^{-1}g_0)$.
\end{prop}
\begin{proof}
The case $n=2$ is easily verified by a direct calculation in $\cC_{2,2}$. Now suppose that the formula \eqref{eq:idemENrenormCn2} is true for some $n\geq 2$. To show that it is true for $n+1$, one can use the recurrence relation \eqref{eq:Erec} for $E_{n+1}^{(q,\alpha_b)}$ in $\cH{n+1}$ and then replace $g_1g_0g_1$ using the defining relation \eqref{eq:relquotient2} of $\cC_{n+1,2}$. With the help of the property $E_{n}^{(q,\alpha_b)}g_0 = \alpha_bE_{n}^{(q,\alpha_b)}$, the result simplifies to
\begin{align}
E_{n+1}^{(q,\alpha_b)} = 
E_{n}^{(q,\alpha_b)}\Big( &1+qg_n+\dots + q^{n-1}g_n\dots g_2 - q^{n+3} g_n \dots g_2(1+qg_2+\dots +q^{n-1}g_2 \dots g_n)  \nonumber \\
 &+ q^ng_n\dots g_1(1+q^{2}+q^3g_2+\dots +q^{n+1}g_2 \dots g_n)(1-\alpha_{b+1}^{-1} g_0) \Big) \ .
\end{align}
The remaining terms can be simplified using the Hecke relation \eqref{eq:cycHrel4}, the braid relations \eqref{eq:cycHrel1} and  \eqref{eq:cycHrel3}, and the property $E_{n}^{(q,\alpha_b)}g_i = qE_{n}^{(q,\alpha_b)}$ for $i=1,\dots,n-1$ to arrive at
\begin{equation}
E_{n+1}^{(q,\alpha_b)} = 
E_{n}^{(q,\alpha_b)}q^{n}[n+1]_q\Big( (1-q^2)(1+qg_n+\dots + q^{n-1}g_n\dots g_2)+ q^{n}g_n\dots g_1(1-\alpha_{b+1}^{-1} g_0) \Big) \ .
\end{equation}
 The proof is completed by using the induction hypothesis. 
\end{proof}


\section{The fused Hecke algebra}
\label{sec:fusedH}

Let $k\geq 1$. We briefly recall the definition of the fused Hecke algebra $\fH$ in the particular case where only the $k$ first strands are fused, and refer to \cite{CP,CP2,P1} for more details. The fused Hecke algebra $H_{k,n}$ is defined for a non-zero complex number $q$ satisfying
\begin{equation}\label{cond-q}
 q^{2i}\neq 1\,,\ \ i=1,\dots,k\ ,
 \end{equation}
since denominators of the form $q^{2i}-1$, with $i=1,\dots,k$, appear in its definition. Equivalently, we will consider $H_{k,n}$ to be defined over the ring generated by $\bC[q^{\pm1}]$ and $(q^{2i}-1)^{-1}$ for $i=1,\dots,k$:
\begin{equation}\label{cond-ring}
\bC^{(k)}[q^{\pm1}]=\bC[q^{\pm1},(q^2-1)^{-1}\dots(q^{2k}-1)^{-1}]\ .
\end{equation}
\begin{rem}
The fused Hecke algebra can also be defined for $q=\pm1$ (and is called the algebra of fused permutations in \cite{CP}), but we will not consider this possibility here, since this would require to replace from the beginning the algebra $\cH{n}$ of Section \ref{sec-A} by a different algebra.
\end{rem}

\subsection{Definition of $\fH$}\label{subsec:deffusedH}

Let $n\geq 0$. We denote $H_{k+n}$ the Hecke algebra associated to the symmetric group $\fS_{k+n}$, with generators $\si_i$, $i=1,\dots,k+n-1$. Its definition was given at the end of Subsection \ref{subsec-cyclo}, in terms of generators $g_i$'s, but we change the names of the generators to avoid any confusion. We define it over $\bC^{(k)}[q^{\pm1}]$.

The standard basis elements of $H_{k+n}$ are denoted $\si_{\omega}$, where $\omega\in\fS_{k+n}$. Recall that $\si_{\omega}=\si_{i_1}\dots\si_{i_l}$ if $\omega=s_{i_1}\dots s_{i_l}$ is a reduced expression in terms of the adjacent transpositions $s_i=(i,i+1)$. 

The (normalised) $q$-symmetriser of the algebra $H_k$ is:
\begin{equation}
	P_{k} = \frac{\sum_{\om\in \fS_k}q^{\ell(\om)}\si_\om}{\sum_{\om\in \fS_k}q^{2\ell(\om)}}=\frac{q^{-k(k-1)/2}}{[k]_q!} \sum_{\om\in \fS_k}q^{\ell(\om)}\si_\om\ ,
\end{equation}
and is well-defined over the ring $\bC^{(k)}[q^{\pm1}]$, see \eqref{cond-ring}. We see the $q$-symmetriser $P_k$ as an element of $H_{k+n}$ through the natural embedding of $\fS_k$ in $\fS_{k+n}$, where $\fS_k$ acts on the first $k$ letters.

The $q$-symmetriser is a primitive central idempotent of $H_k$: we have $ P_k\si_i = \si_iP_k=qP_k$ for $i=1,\dots,k-1$ and $P_kP_i=P_iP_k=P_k$ for all $i\leq k$. It satisfies the recursive formula:
\begin{equation}
	 q^{k-1}[k]_q P_k = P_{k-1}(1+q\si_{k-1}+q^2\si_{k-1}\si_{k-2}+ \cdots + q^{k-1}\si_{k-1}\si_{k-2}\dots \si_{1}) \ . \label{eq:PkPk-1}
\end{equation}
	
\begin{defi}
The fused Hecke algebra $\fH$ is the algebra of the form $P_kH_{k+n}P_k$.
\end{defi}

A basis of the algebra $\fH$ is indexed by the double cosets $\cosets$ of the subgroup $\fS_k$ in $\fS_{k+n}$. In each of these cosets, there is a unique representative of minimal length (see \cite{GP}), and we will identify the set of double cosets $\cosets$ with the set of minimal-length representatives. The standard basis of $\fH$ is:
\[\{P_k\si_{\omega}P_k\ |\ \omega\in\cosets\}\ .\]
The dimension of the algebra $\fH$ is the number of double cosets in $\cosets$, or of what was called fused permutations in \cite{CP}. It is the number of ways to connect a row of one ellipse and $n$ dots to another such row, with the requirement that $k$ edges start and $k$ edges arrive at the ellipses (for the dots, it is the usual rule of a single edge at each dot). To count such objects, one has first to choose how many edges from, say, the top ellipse will go to a dot. This is choosing $i \in \{1,2,\dots,\min \{k,n\}\}$. Then one needs to chose $i$ bottom dots among $n$ where to put these $i$ edges and  $i$ top dots among $n$ which will be connected to the bottom ellipse. Finally, one can choose an arbitrary permutation diagram between the remaining two lines of $n-i$ dots which are not connected to the ellipses. It follows from this discussion that
\begin{equation}
	\dim (\fH) = \sum_{i=0}^{\min \{k,n\}} (n-i)!\binom{n}{i}^2. \label{eq:dimfH}
\end{equation}  

For any word $x$ in the generators $\si_i$ of $H_{k+n}$ and their inverses $\si_i^{-1}$, the diagrammatic representation of the element $P_kxP_k$ of $\fH$ is obtained by drawing the usual braid-like picture for $x$ between two rows of $k+n$ dots, and then fusing in one large dot (or ellipse) the $k$ first top dots and similarly for the $k$ first bottom dots. For example, we define the following elements of $\fH$ (for $i=1,2,\dots,n-1$):
\begin{align}
	&S_i := 
	\begin{tikzpicture}[scale=0.3,baseline={([yshift=\eseq]current bounding box.center)}]
		\draw[white] (-1.4,0.8) arc (90:270:0.8);
		\ellk{0}{0}
		\ellstrand{0}{0}
		\slab{(4,3)}{$1$}
		\strand{4}{2}
		\node at (6,0) {$\dots$};
		\slab{(8,3)}{$i-1$}
		\strand{8}{2}
		\slab{(12,3)}{$i$}
		\slab{(16,3)}{$i+1$}
		\ocross{12}{2}
		\slab{(20,3)}{$i+2$}
		\strand{20}{2}
		\node at (22,0) {$\dots$};
		\slab{(24,3)}{$n$}
		\strand{24}{2}
	\end{tikzpicture}, \label{eq:defSi}\\
	&S_0 :=
	\begin{tikzpicture}[scale=0.3,baseline={([yshift=\eseq]current bounding box.center)}]
		\ellk{0}{0}
		\ellstrand{0}{0}
		\slab{(4,3)}{$1$}
		\encell{4}{2}
		\slab{(8,3)}{$2$}
		\strand{8}{2}
		\node at (10,0) {$\dots$};
		\slab{(12,3)}{$n$}
		\strand{12}{2}
	\end{tikzpicture}, \label{eq:defS0}\\
	&T :=
	\begin{tikzpicture}[scale=0.3,baseline={([yshift=\eseq]current bounding box.center)}]
		\draw[white] (-1.4,0.8) arc (90:270:0.8);
		\ellk{0}{0}
		\ellocross{0}{0}
		\slab{(4,3)}{$1$}
		\slab{(8,3)}{$2$}
		\strand{8}{2}
		\node at (10,0) {$\dots$};
		\slab{(12,3)}{$n$}
		\strand{12}{2}
	\end{tikzpicture}\ , \label{eq:defT} 
\end{align}
which algebraically correspond to (for $i=1,2,\dots,n-1$)
\begin{align}
	&S_i = P_k\si_{k+i}P_k = P_k\si_{k+i} = \si_{k+i}P_k, \label{eq:Sialg}\\
	&S_0 = P_k\si_{k}\si_{k-1}\dots \si_{2}\si_{1}^2\si_{2}\dots\si_{k-1}\si_{k}P_k, \label{eq:S0alg}\\
	&T = P_k\si_{k}P_k. \label{eq:Talg}
\end{align}
We have used that $\si_{k+i}$ commutes with $P_k$ when $i\geq 1$.

\paragraph{Eigenvalues of the generator $S_0$.} We will need to know the characteristic equation satisfied by $S_0$ in order to relate the algebra $\fH$ to a cyclotomic quotient $\cH{n}$ for the correct values of $\alpha_1,\alpha_2$.
\begin{prop}\label{prop-eigS0}
We have in $\fH$:
\begin{equation}
	\left(S_0-q^{2k}P_k\right)\left(S_0-q^{-2}P_k\right)=0\ .
\end{equation}
\end{prop}
\begin{proof}
First, we show the following relation between $S_0$ and $T$:
\begin{align}
	S_0 = (q-q^{-1})q^{k-1}[k]_q T +P_k. \label{eq:relST}
\end{align}
To check this, we start with the defining formula for $S_0$ and use the quadratic relation for $\si_1$. We find:
\[S_0=(q-q^{-1})P_k\sigma_{k}\sigma_{k-1}\dots \si_2\si_1\si_2\dots\si_{k-1}\si_k P_k+P_k\sigma_{k}\sigma_{k-1}\dots \si_2^2\dots\si_{k-1}\si_k P_k\ .\] 
Using the braid relations and the property of $P_k$, the first term becomes
\[(q-q^{-1})P_k\sigma_{1}\sigma_{2}\dots \si_{k-1}\si_k\si_{k-1}\dots\si_2\si_1 P_k=(q-q^{-1})q^{2(k-1)}P_k\si_kP_k=(q-q^{-1})q^{2(k-1)}T\ .\]
We proceed similarly with the remaining term, which ultinately gives Formula \eqref{eq:relST}.
 
The proof of the proposition is concluded by calculating the eigenvalues of $T$. We have:
\begin{equation}
	\left(T-qP_k\right)\left(T+q^{-k}[k]_q^{-1}P_k\right)=0. \label{eq:Tchar}
\end{equation}
To check this equality, we first use the recurrence relation for $P_k$ to write:
\[\si_kP_k\si_k=q^{-k+1}[k]_q^{-1}P_{k-1}\si_k(1+q\si_{k-1}+q^2\si_{k-1}\si_{k-2}+\dots+q^{k-1}\si_{k-1}\dots \si_1)\si_k\ .\]
Then we proceed as follows, using $P_kP_{k-1}=P_k$, the braid relations and $\si_iP_k=q P_k$ if $i<k$:
	\begin{align}
	T^2=  P_k\si_kP_k\si_kP_k & =q^{-k+1}[k]_q^{-1}P_k(1+(q+q^3+q^5+ \cdots + q^{2k-1}-q^{-1})\si_k)P_k \\
	&= q^{-k+1}[k]_q^{-1}P_k+(q-q^{-k}[k]_q^{-1})T\ .
\end{align}
This concludes the verification.
\end{proof}

\paragraph{The semisimple representation theory of $\fH$.} In this paragraph, we work over $\bC(q)$ (or we assume that $q^2$ is not a root of unity of order $\leq k+n$).

The irreducible representations of $\bC(q)H_{k,n}$ were described in \cite{CP}. They are indexed by partitions $\lambda\vdash k+n$ such that $\lambda_1\geq k$ (that is, the first line of $\lambda$ contains at least $k$ boxes). For example, for $n=0$, there is a single irreducible representation, indexed by a line of $k$ boxes. The branching rules are given by inclusion of partitions. For example, when $k=3$, the beginning of the Bratteli diagram is:
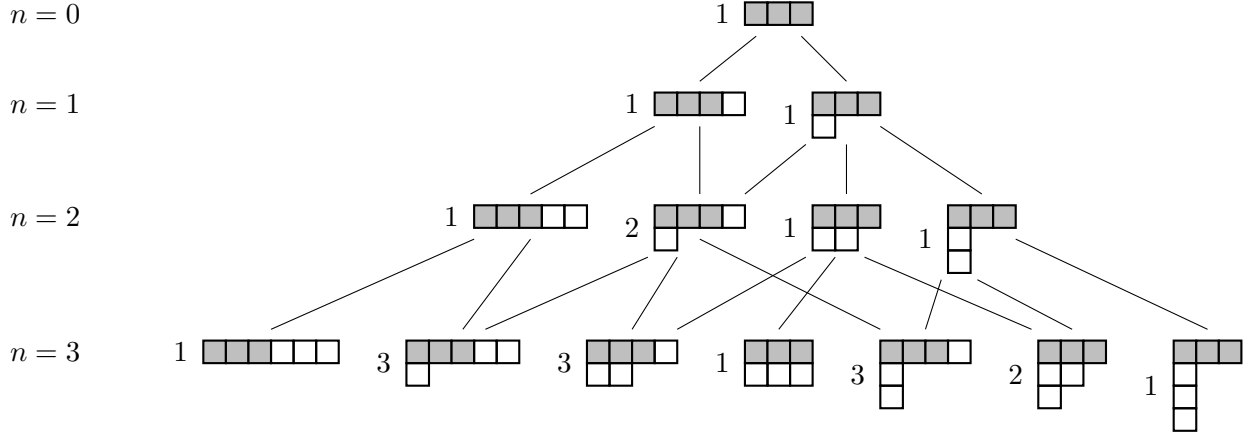
\begin{center}
 \begin{tikzpicture}[scale=0.3]
\mdiagp{-1}{0}{3};\node at (-2,-0.5) {$1$};
\draw (-0.5,-1.5) -- (-3,-3.5);\draw (1.5,-1.5) -- (3.5,-3.5);
\mdiag{-5}{-4}{4};\node at (-6,-4.5) {$1$};\mdiaggp{2}{-4}{3}{1};\node at
(1,-5) {$1$};

\draw (-5,-5.5) -- (-10.5,-8.5);\draw (-3,-5.5) -- (-3,-8.5);\draw
(1.7,-6.3) -- (-1,-8.5);\draw (3.5,-6.3) -- (3.5,-8.5);\draw (5,-5.5)
-- (9.5,-8.5);

\mdiag{-13}{-9}{5};\node at (-14,-9.5) {$1$};\mdiagg{-5}{-9}{4}{1};\node
at (-6,-10) {$2$};\mdiaggp{2}{-9}{3}{2};\node at (1,-10)
{$1$};\mdiagggp{8}{-9}{3}{1}{1};\node at (7,-10.5) {$1$};

\draw (-13,-10.5) -- (-22,-14.5);\draw (-10.5,-10.5) --
(-13.5,-14.5);\draw (-5.3,-11.3) -- (-12.5,-14.5);\draw (-4,-11.3) --
(-6,-14.5);\draw (-3,-10.5) -- (5,-14.5);
\draw (1.7,-11.3) -- (-4,-14.5);\draw (3,-11.3) -- (0.5,-14.5);\draw
(4.3,-11.3) -- (11.7,-14.5);\draw (7.7,-12.3) -- (7,-14.5);\draw
(9.3,-12.3) -- (13.5,-14.5);
\draw (11,-10.5) -- (19.5,-14.5);

\node at (-26,-15.5) {$1$};\mdiag{-25}{-15}{6};
\node at (-17,-16) {$3$};\mdiagg{-16}{-15}{5}{1};

\node at (-9,-16) {$3$};\mdiagg{-8}{-15}{4}{2};

\node at (-2,-16) {$1$};\mdiaggp{-1}{-15}{3}{3};

\node at (4,-16.5) {$3$};\mdiaggg{5}{-15}{4}{1}{1};\node at (11,-16.5)
{$2$};\mdiagggp{12}{-15}{3}{2}{1};\node at (17,-17)
{$1$};\mdiaggggp{18}{-15}{3}{1}{1}{1};


\node at (-32,-0.5) {$n=0$};\node at (-32,-4.5) {$n=1$};\node at
(-32,-9.5) {$n=2$};\node at (-32,-15.5) {$n=3$};
\end{tikzpicture}
\end{center}
We have shaded the three fixed boxes in the first row of each partition. Next to each partition is the dimension of the corresponding irreducible representation. We emphasize that the dimension is not the number of standard tableaux strictly speaking, but is the number of standard fillings of the non-shaded boxes by $1,\dots,n$.

It is easy to see that, when $2\leq n\leq k$, there are three one-dimensional representations for $H_{k,n}$. They are given by the following partitions:
\begin{equation}\label{1dimfH}
\lambda=(k+n)\ :\begin{array}{l} S_0\mapsto q^{2k}\,,\\
S_i\mapsto q\,, \end{array}\qquad 
\lambda=(k,1^n)\ :\begin{array}{l} S_0\mapsto q^{-2}\,,\\
S_i\mapsto -q^{-1}\,,\end{array} \qquad
\lambda=(k,n)\ :\begin{array}{l} S_0\mapsto q^{-2}\,,\\
S_i\mapsto q\,,\end{array}
\end{equation}
where we give the associated values of the elements $S_0$ and $S_i$ ($i\geq 1$). These are easily obtained from the description in \cite{CP}. The case of $\lambda=(k+n)$ is immediate and the level $n=1$ together with the given branching rules give all the other values of $S_0$. Calculating the eigenvalue of $S_i$ is also immediate from the description in \cite{CP} since $S_i=P_k\sigma_{k+i}$ with $\sigma_{k+i}$ commuting with $P_k$.

When $n>k$, there is only two remaining one-dimensional representations, the one corresponding to $\lambda=(k,n)$ being removed from the list (it would not make sense for $n>k$).

\begin{rem}\label{rem-repfH}
Let $n\leq k$. Anticipating a little bit, the algebra $\fH$ is going to be obtained (below) as a specialisation of the parameters $\alpha_1,\alpha_2$ of the algebra $\cA_n$. This specialisation is semisimple over $\bC(q)$, and thus we have a bijection between the irreducible representations of $\bC(q)\fH$ and the representations described in Section \ref{sec-A}:
\[Irr(F\cA_{n})=\{V_{(\lambda,\mu)}\ |\ (\lambda,\mu)\in Par_2(n)\ \text{with}\ \ell(\mu)\leq 1\ \}\ .\]
The bijection maps $(\lambda,\mu)$ to the partition made of a first line of $k+|\mu|$ boxes, and with the diagram of $\lambda$ for the remaining lines. This is well-defined since $n$, and in turn $|\lambda|$, is less or equal to $k$.

If $n>k$, the specialisation is not semisimple anymore and the algebra $\fH$ will be obtained as a quotient of this non-semisimple specialisation of $\cA_n$. Therefore we can not immediately identify the irreducible representations of $\fH$ with a subset of the irreducible representations of $\cA_n$ as soon as $n>k$.
\end{rem}

\subsection{An algebraic description of $\fH$}

In this section, we use the algebra $\cA_n$ from Section \ref{sec-A} to obtain an algebraic presentation of the fused Hecke algebra $\fH$. We are going to work in this section with the following specialisation of the parameters $\alpha_1,\alpha_2$:
\begin{equation}
	\alpha_1=q^{-2}\ , \quad \alpha_2=q^{2k}\ . \label{eq:param2}
\end{equation}
This is motivated by Proposition \ref{prop-eigS0} since these two values are the eigenvalues of the element $S_0$. For these specific values, we have the following factorisation result for one of the quasi-idempotents of $\cA_n$, as a corollary of our results in Section \ref{subsec-idemA}.
\begin{coro}\label{coro:Ek+1}
In $\cA_{k+1}$, when $(\alpha_1,\alpha_2)$ are specialised as in \eqref{eq:param2}, the element $E_{k+1}^{(q,\alpha_1)}$ factorises as:
\[E_{k+1}^{(q,\alpha_1)}=[k+1]_q {E'}_{k+1}^{(q,\alpha_1)}\,,\]
where ${E'}_{k+1}^{(q,\alpha_1)}$ is in $\cA_{k+1}$ with coefficients in $\bC[q^{\pm1}]$.
\end{coro}
\begin{proof}
It is clear that $E_{k+1}^{(q,\alpha_1)}-E_{k+1}^{(q,\alpha_2)}$ is divisible by $(\alpha_1-\alpha_2)$. With the given specialisation, this means that it is divisible by $q^{2(k+1)}-1$ and thus by $[k+1]_q$. Now Proposition \ref{prop-idempAn} shows that $E_{k+1}^{(q,\alpha_2)}$ is also divisible by $[k+1]_q$ (actually, already over $R$). Thus we get the desired statement.
\end{proof}
This allows us to define, in the specialisation of $\cA_{k+1}$, the following element with coefficients in $\mathbb{C}[q^{\pm1}]$:
\[{E'}_{k+1}^{(q,\alpha_1)}:= \frac{1}{[k+1]_q}E_{k+1}^{(q,\alpha_1)}\ .\]
\begin{defi}\label{def-Ank}
The algebra $\cA_n^{(k)}$ is defined over $\bC^{(k)}[q^{\pm1}]$ as the specialisation of $\cA_n$ corresponding to \eqref{eq:param2} with the additional defining relation if $n>k$:
\begin{equation}
	{E'}_{k+1}^{(q,\alpha_1)}=0\ . \label{eq:relidem2}
\end{equation}  
\end{defi}
\begin{rem}
The additional relation when $n\geq k+1$ is the analogue in the Hecke case of the additional relation that is needed to pass from the one-boundary Temperley--Lieb algebra to the boundary seam algebra. We refer to \cite{MRR} and \cite{L-L} where this additional relation was interpreted in terms of a quasi-idempotent of the one-boundary Temperley--Lieb algebra (see also Section \ref{subsecboundaryseam}).
\end{rem}

\paragraph{A spanning set of $\cA_n^{(k)}$.} First, we find a spanning set for $\cA_n^{(k)}$. Consider the signed permutations with the following avoiding patterns: $\bar{1}\bar{2}$ and $\overline{k+1}\bar{k}\dots\bar{1}$. We denote by $B_n(\bar{1}\bar{2},\overline{k+1}\bar{k}\dots\bar{1})$ the subset of these signed permutations.

The first condition defining $B_n(\bar{1}\bar{2},\overline{k+1}\bar{k}\dots\bar{1})$ is the same as the one giving the basis of $\cA_n$ in Theorem \ref{theo:basisA}. The second one adds the condition that in the word $b_1b_2\dots b_n$ giving the signed permutations, we never have a strictly decreasing sequence of length $k+1$ of barred numbers. Both together mean that all barred numbers appear in decreasing orders, and that at most $k$ barred numbers are present. So to count the number of elements in $B_n(\bar{1}\bar{2},\overline{k+1}\bar{k}\dots\bar{1})$, we reason as follows on the word $b_1b_2\dots b_n$ giving the signed permutation. We choose $i$ numbers in $\{1,\dots,n\}$ that will be barred, with the condition that $i\leq k$; then choose $i$ positions among $n$ to place these barred numbers in decreasing order; and finally we choose a permutation of the remaining $n-i$ numbers to be placed (not barred) in the remaining $n-i$ positions. This gives:
\begin{equation}
	|B_n(\bar{1}\bar{2},\overline{k+1}\bar{k}\dots\bar{1})| = \sum_{i=0}^{\text{min}(k,n)}  \binom{n}{i}^2(n-i)! . \label{eq:cardavoid-k}
\end{equation}
Note that the above discussion shows the alternative description:
\begin{equation}\label{Bn12k}
B_n(\bar{1}\bar{2},\overline{k+1}\bar{k}\dots\bar{1})=B_n(\bar{1}\bar{2})\cap \{\omega\in B_n\ |\ \ell_0(\omega)<k+1\}\ .
\end{equation}
Let us emphasize that if $n\leq k$, then the set $B_n(\bar{1}\bar{2},\overline{k+1}\bar{k}\dots\bar{1})$ is the same as $B_n(\bar{1}\bar{2})$ since there are no signed permutations in $B_n$ with more than $k$ barred numbers.
\begin{prop}\label{prop-spanAnk}
The following set is a spanning set of $\cA_n^{(k)}$:
\begin{equation}
\{g_\omega\ |\ \omega\in B_n(\bar{1}\bar{2},\overline{k+1}\bar{k}\dots\bar{1})\}\ .
\label{eq:spanAnk}
\end{equation}
\end{prop}
\begin{proof}
Recall the basis $\{g_\omega\ |\ \omega\in B_n(\bar{1}\bar{2})\}$ of $\cA_n$ from Theorem \ref{theo:basisA}. There is a unique basis element containing $n$ occurrences of $g_0$ and all the others contain strictly less. In other words, there is a unique element $\omega\in B_n(\bar{1}\bar{2})$ with maximal $\ell_0$, which is $\ell_0(w)=n$. This element corresponds to the sequence $\bar{n},\overline{n-1},\dots,\bar{1}$, and the corresponding basis element is:
\begin{equation}\label{rewriting-k}
g_0\cdot g_1g_0\cdot\ldots\ldots\cdot g_{n-1}\dots g_1g_0\,.
\end{equation}
\begin{lem}
The element \eqref{rewriting-k} appears with coefficient $(-1)^n\alpha_2^{-n}q^{(n-1)n}[n]_q!$ when $E_{n}^{(q,\alpha_1)}$ is expanded in the basis $\{g_\omega\ |\ \omega\in B_n(\bar{1}\bar{2})\}$ of $\cA_n$.
\end{lem}
\begin{proof}[Proof of the lemma] The proof is by induction on $n$. The statement is immediate when $n=1$. Then we use the recurrence formula for the quasi-idempotent:
\[E^{(q,\alpha_1)}_{n+1} = E^{(q,\alpha_1)}_{n}\bigl(1+\sum_{i=1}^{n}q^{n-i+1}g_{n}\dots g_{i}-\alpha_2^{-1}q^n g_{n}\dots g_1g_0(1+\sum_{i=1}^{n}q^ig_1\dots g_i) \bigr)\ .\]
Using the induction hypothesis and keeping the only terms which can contribute to the relevant coefficient, we have to study:
\[-\gamma_{n}\alpha_2^{-1}q^{n}g_0\cdot g_1g_0\cdot\ldots\ldots\cdot g_{n-1}\dots g_1g_0\cdot g_{n}\dots g_1g_0(1+\sum_{i=1}^{n}q^ig_1\dots g_i)\,,\]
where $\gamma_n$ is the coefficient given by the induction hypothesis. For dealing with the terms in the sum, note that they produce the appearance of $g_0g_1g_0g_1$ (the leftmost $g_0$ is the $g_0$ just to the left of $g_n$). So we can use the defining relation \eqref{eq:relquotient} of $\cA_n$ and keep only the term not reducing the occurrences of $g_0$. This amounts to simply replacing $g_0g_1g_0g_1$ by $qg_0g_1g_0$. Thus all $g_1$'s in the sum are replaced by $q$. Similarly, if an element $g_i$ hits the expression on the left of the parenthesis, it will move through $g_n\dots g_1g_0$, $g_{n-1}\dots g_1g_0$, $...$, becoming successively $g_{i-1}$, $g_{i-2}$ and so on until it becomes $g_1$ and hits again a $g_0g_1g_0$. The same reasoning as before allows to replace it by $q$. So all generators $g_1,\dots,g_n$ appearing in the parenthesis are replaced by $q$ and we are left with $(1+q^2+\dots+q^{2n})=q^n[n+1]_q$. This proves the lemma.
\end{proof}
The lemma implies that in the renormalised quasi-idempotent ${E'}_{k+1}^{(q,\alpha_1)}$, the coefficient in front of the element \eqref{rewriting-k} with $n=k+1$ is
\[(-1)^{k+1}\alpha_2^{-(k+1)}q^{k(k+1)}[k]_q!\ ,\]
since we have divided ${E}_{k+1}^{(q,\alpha_1)}$ by $[k+1]_q$. This coefficient is invertible in the ring $\bC{(k)}[q^{\pm1}]$.

Now we will show that every element $g_{\omega}$ with $\ell_0(\omega)\geq k+1$ can be rewritten in $\cA_n^{(k)}$ in terms of elements $g_{\omega'}$ with $\ell_0(\omega')< k+1$. Indeed, if the number of occurrences of $g_0$ is at least $k+1$ then write $g_{\omega}$ in the standard form \eqref{eq:basecH} and pick $k+1$ consecutive occurrences of $g_0$. Then note using the braid relations that
\[g_j\cdot g_i\dots g_1g_0=g_i\dots g_1g_0\cdot g_{j+1}\ \ \ \ \forall 0<j<i\ .\]
This allows to find a subexpression of $g_{\omega}$ which is
\[g_{i_1}\dots g_1g_0\cdot g_{i_2}\dots g_1g_0\cdot\ldots\ldots\cdot g_{i_{k+1}}\dots g_1g_0\,,\ \ \ \ \ \ \text{with $i_1<i_2\dots <i_{k+1}$.}\]
Moving some $g_i$'s to the left using commutation relations, we see at once that the element \eqref{rewriting-k} with $n=k+1$ appears. Thanks to the lemma and its consequence stated just after, we can use the relation ${E'}_{k+1}^{(q,\alpha_1)}=0$ of $\cA_n^{(k)}$ to rewrite this element in terms of $g_{\omega'}$ with $\ell_0(\omega')< k+1$. Thus we have strictly reduced the number of occurrences of $g_0$, and by induction we conclude that $\cA_n^{(k)}$ is spanned by elements $g_{\omega'}$ with $\ell_0(\omega')< k+1$.

Now if an element $\omega$ with $\ell_0(\omega)<k+1$ has the pattern $\bar{1}\bar{2}$ then as shown in the proof of Theorem \ref{theo:basisA} we can use the defining relation of $\cA_n$ allowing to rewrite $g_0g_1g_0g_1$ to write $g_\omega$ in terms of $g_{\omega'}$ with $\omega'\in B_n(\bar{1}\bar{2})$. In doing so, note that the 0-length $\ell_0$ never increases since the relation we use rewrites $g_0g_1g_0g_1$ in terms of elements with $2$ occurences of $g_0$ or less.
\end{proof}

\paragraph{Isomorphism theorem.} We are now ready to state the main result of this section.
\begin{thm}\label{thm-iso1}
For any $n$, there is an algebra isomorphism from $\cA^{(k)}_n$ to $\fH$ given by:
 \begin{equation}\label{prop:homo}
	\begin{array}{rl}
		\phi\ :\ \cA^{(k)}_{n} &\to \fH \\
		1 &\mapsto P_k\\
		g_i &\mapsto S_i \quad i=0,1,\dots,n-1\, .
	\end{array}
	\end{equation}
\end{thm}

\begin{coro}\label{coro-basis}
Over $\bC^{(k)}[q^{\pm1}]$, the following set is a basis of $\cA_n^{(k)}$ and its image by $\phi$ is a basis of $\fH$:
\begin{equation*}
\{g_\omega\ |\ \omega\in B_n(\bar{1}\bar{2},\overline{k+1}\bar{k}\dots\bar{1})\}\ .
\end{equation*}
\end{coro}
\begin{proof}
This  follows from the fact that this set is a spanning set (Proposition \ref{prop-spanAnk}) with cardinality given by \eqref{eq:cardavoid-k}, which is equal to the dimension of $\fH$, and from the isomorphism of $\cA_n^{(k)}$ with $\fH$.    
\end{proof}

\subsection{Proof of Theorem \ref{thm-iso1}}

\paragraph{The morphism property of $\phi$.} By definition, the projector $P_k$ acts as the unit element in $\fH$. We must verify that the elements $S_i \in \fH$ for $i=0,1,\dots,n$ satisfy the same defining relations \eqref{eq:cycHrel1}--\eqref{eq:cycHrel5}, \eqref{eq:relquotient} and \eqref{eq:relidem2} as the elements $g_i \in \cA^{(k)}_{n}$. Relations \eqref{eq:cycHrel1}, \eqref{eq:cycHrel3} for $i,j\neq 0$ and \eqref{eq:cycHrel4}  directly follow from the definition of the fused Hecke algebra. It is also readily apparent that \eqref{eq:cycHrel3} for $i=0$ is verified, since $S_0$ and $S_j$ with $j\geq 2$ are elements that each act on different strands. In diagrams, we have (ignoring unaffected strands in the illustrations for simplicity):
\begin{equation}
\begin{tikzpicture}[scale=0.3,baseline={([yshift=\eseq]current bounding box.center)}]
			\ellk{0}{0}
			\ellstrand{0}{0}
			\slab{(4,3)}{$1$}
			\encell{4}{2}
			\slab{(8,3)}{$i$}
			\strand{8}{2}
			\slab{(12,3)}{$i+1$}
			\strand{12}{2}
			\begin{scope}[yshift=-4cm]
			\ellk{0}{0}
			\ellstrand{0}{0}
			\strand{4}{2}
			\node at (6,2) {$\dots$};
			\ocross{8}{2}
			\end{scope}
\end{tikzpicture} = 
\begin{tikzpicture}[scale=0.3,baseline={([yshift=\eseq]current bounding box.center)}]
			\ellk{0}{0}
			\ellstrand{0}{0}
			\slab{(4,3)}{$1$}
			\encell{4}{2}
			\slab{(8,3)}{$i$}
			\slab{(12,3)}{$i+1$}
			\node at (6,0) {$\dots$};
			\ocross{8}{2}
\end{tikzpicture} = 
\begin{tikzpicture}[scale=0.3,baseline={([yshift=\eseq]current bounding box.center)}]
			\ellk{0}{0}
			\ellstrand{0}{0}
			\slab{(4,3)}{$1$}
			\strand{4}{2}
			\slab{(8,3)}{$i$}
			\slab{(12,3)}{$i+1$}
			\ocross{8}{2}
			\begin{scope}[yshift=-4cm]
			\ellk{0}{0}
			\ellstrand{0}{0}
			\encell{4}{2}
			\strand{8}{2}
			\strand{12}{2}
			\node at (6,2) {$\dots$};
			\end{scope}
\end{tikzpicture}.
\end{equation}  
	
	Consider now relation \eqref{eq:cycHrel2}. We start by computing the following (again ignoring unaffected strands):
	\begin{align}
		S_1S_0S_1 &= 
		\begin{tikzpicture}[scale=0.3,baseline={([yshift=\eseq]current bounding box.center)}]
			\ellk{0}{0}
			\ellstrand{0}{0}
			\slab{(4,3)}{$1$}
			\slab{(8,3)}{$2$}
			\ocross{4}{2}				
			\begin{scope}[yshift=-4cm]
			\ellk{0}{0}
			\ellstrand{0}{0}
			\encell{4}{2}
			\strand{8}{2}
			\end{scope}			
			\begin{scope}[yshift=-8cm]
			\ellk{0}{0}
			\ellstrand{0}{0}
			\ocross{4}{2}
			\end{scope}			
		\end{tikzpicture}
		=
		\begin{tikzpicture}[scale=0.3,baseline={([yshift=\eseq]current bounding box.center)}]
			\ellk{0}{0}
			\ellstrand{0}{0}
			\slab{(4,3)}{$1$}
			\strand{4}{2}
			\slab{(8,3)}{$2$}
			\encells{8}{2}					
		\end{tikzpicture} .
	\end{align}
Now, in order to show that \eqref{eq:cycHrel2} holds for the images,  one must show that $S_0$ commutes with $S_1S_0S_1$. This is seen using the homotopy of diagrams and the fact that an ellipse can be replaced by a sum of braids acting on the $k$ fused strands. This means that the strand labeled $2$, which encircles both the strand $1$ and the $k$ fused strands, can be moved up or down around any middle ellipse and around all encircled strands, hence the commutativity. In diagrams, we have:  
\begin{align}
		\begin{tikzpicture}[scale=0.3,baseline={([yshift=\eseq]current bounding box.center)}]
			\draw[white] (-1.4,1) arc (90:270:1.2);
			\ellk{0}{0}
			\ellstrand{0}{0}
			\slab{(4,3)}{$1$}
			\slab{(8,3)}{$2$}
			\encell{4}{2}
			\strand{8}{2}				
			\begin{scope}[yshift=-4cm]
			\ellk{0}{0}
			\ellstrand{0}{0}
			\strand{4}{2}
			\encells{8}{2}
			\end{scope}		
		\end{tikzpicture} =
		\begin{tikzpicture}[scale=0.3,baseline={([yshift=\eseq]current bounding box.center)}]
			\ellk{0}{0}
			\ellstrand{0}{0}
			\slab{(4,3)}{$1$}
			\slab{(8,3)}{$2$}
			\encell{4}{2}
			\encellsu{8}{2}					
		\end{tikzpicture} = 
		\begin{tikzpicture}[scale=0.3,baseline={([yshift=\eseq]current bounding box.center)}]
			\draw[white] (-1.4,1) arc (90:270:1.2);
			\ellk{0}{0}
			\ellstrand{0}{0}
			\slab{(4,3)}{$1$}
			\slab{(8,3)}{$2$}
			\strand{4}{2}
			\encells{8}{2}					
			\begin{scope}[yshift=-4cm]
			\ellk{0}{0}
			\ellstrand{0}{0}
			\encell{4}{2}
			\strand{8}{2}	
			\end{scope}		
		\end{tikzpicture} .
	\end{align}
	
The quadratic relation \eqref{eq:cycHrel5}  for $g_0$ is preserved thanks to Proposition \ref{prop-eigS0}.

Then to show that Relation \eqref{eq:relquotient} is preserved, we start with a lemma. 
\begin{lem} \label{lem:TSTS}
For $k$ any positive integer, the following relations hold
\begin{align}
	&TS_1TS_1 -qTS_1T
	=q^{1-k}[k]_q^{-1}(S_1T-q^{-1}S_1TS_1), \label{eq:TSTS}\\ 
	&S_1TS_1T -qTS_1T
	=q^{1-k}[k]_q^{-1}(TS_1-q^{-1}S_1TS_1). \label{eq:STST}
\end{align}
\end{lem}
\begin{proof}
	Similarly as has been done before in this section, one can use the algebraic expressions of $T$ and $S_1$, the properties of the projector $P_k$ as well as the braid relations to show that 
	\begin{align}
	TS_1T 
	= P_k\si_k\si_{k+1}P_k\si_kP_k
	= q^{1-k}[k]_q^{-1}(S_1TS_1+q^{k-1}[k-1]_qP_k\si_k\si_{k-1}\si_{k+1}\si_kP_k).
\end{align}
Multiplying the previous equation by $S_1$ on the right, and then using the Hecke relation and braid relations, it is found that
\begin{align}
	TS_1TS_1 
	&= q^{1-k}[k]_q^{-1}((q-q^{-1})S_1TS_1+S_1T+q^{k}[k-1]_qP_k\si_{k}\si_{k-1}\si_{k+1}\si_kP_k).
\end{align}  
Therefore, combining the two previous results, relation \eqref{eq:TSTS} is found. To obtain \eqref{eq:STST}, multiply by $S_1$ on the left instead.
\end{proof}
One can now use \eqref{eq:relST} to write equation \eqref{eq:TSTS} in terms of $S_0$, which gives
\begin{equation}
	S_0S_1S_0S_1
	=-q^{-2} P_k+q^{-3} S_1+S_0-q^{-1} (S_0 S_1+S_1S_0)+q^{-2}S_1S_0S_1+qS_0S_1S_0.
\end{equation}
The previous relation corresponds to \eqref{eq:relquotient} with parameters $\alpha_1$ and $\alpha_2$ as in \eqref{eq:param2}.

Finally, when $n>k$, we must show that the additional relation \eqref{eq:relidem2} is satisfied in $\fH$. If ${E'}_{k+1}^{(q,\alpha_1)}$ were non-zero, this would imply that this element would also be non-zero in the algebra $\fH$ extended over $\bC(q)$. This in turn would imply the existence of a one-dimensional representation $S_0\mapsto q^{-2}$ and $S_i\mapsto q$, $i\geq 1$. We already discussed the non-existence of such a one-dimensional representation around \eqref{1dimfH}.

\paragraph{Surjectivity of $\phi$.} For $i=1,2,\dots,\min \{k,n\}$, we define the element $U_{i} \in \fH$ which consists of the diagram where the $i$ last strands of the top ellipse go out to the $i$ first bottom circles (without crossing each other), and similarly for the strands of the bottom ellipse. It is illustrated as follows:  
\begin{align}
	&U_{i} :=
	\begin{tikzpicture}[scale=0.3,baseline={([yshift=\eseq]current bounding box.center)}]
		\draw[white] (-1.4,0.8) arc (90:270:0.8);
		\ellU{0}{0}
		\node at (6,0) {$\dots$};
		\slab{(16,3)}{$i+1$}
		\strand{16}{2}
		\node at (18,0) {$\dots$};
		\slab{(20,3)}{$n$}
		\strand{20}{2}
		\slab{(4,3)}{$1$}
		\slab{(8,3)}{$i-1$}
		\slab{(12,3)}{$i$}
	\end{tikzpicture}, \quad i=1,2,\dots,\min \{k,n\}.
\end{align}
Here it is understood that there are $k-i$ straight strands in the gray zone. If we denote
\begin{equation}
	\si_{k,i} := \si_{k}\si_{k+1}\dots \si_{k+i-1},
\end{equation}
then the algebraic expression of the element $U_i$ is given by
\begin{equation}
	U_i = P_k\si_{k,i}\si_{k-1,i}\dots \si_{k-i+1,i} P_k. \label{eq:Uialg}
\end{equation} 


The elements $P_k$, $U_i$ for $i=1,2,\dots,\min \{k,n\}$ and $S_i$ for $i=1,2,\dots,n-1$ generate the fused Hecke algebra $\fH$. Indeed, a basis for $\fH$ consists of the set of fused braid diagrams corresponding to fused permutations, which were described when counting the dimension \eqref{eq:dimfH} of $\fH$. A generic fused braid diagram can be obtained by multiplying the element $U_i$ on the left and on the right by elements $S_j$ for $j=1,\dots,n-1$.

We are now ready to prove that the map $\phi$ is surjective. It suffices to show that the generators of $\fH$ belong to the image of $\phi$. We already know that $P_k$ and $S_i$ for $i=1,\dots,n-1$ belong to the image by definition of the map $\phi$. We will show that it is also the case for the elements $U_i$ by induction on $i$.
	
For $i=1$, we have $U_1=T$, which belongs to the image because of \eqref{eq:relST}. Suppose now that $U_i$ belongs to the image for some integer $1\leq i < \min \{k,n\}$. Then, the following element of $\fH$ also belongs to the image of $\phi$:    
\begin{equation}
	TS_{1}S_{2}\dots S_{i}U_{i} =
	\begin{tikzpicture}[scale=0.3,baseline={([yshift=\eseq]current bounding box.center)}]
		\draw[white] (-1.4,0.8) arc (90:270:0.8);
		\ellTii{0}{0}
		\draw[thick] (0,2) -- (0,-2);
		\draw[thick] (0.5,2) -- (0.5,-2);
		\node at (6,0) {$\dots$};
		\node at (22,-2) {$\dots$};
		\slab{(20,3)}{$i+2$}
		\strand{20}{2}
		\slab{(24,3)}{$n$}
		\strand{24}{2}
		\begin{scope}[yshift=-4cm]
		\draw[white] (-1.4,0.8) arc (90:270:0.8);
		\ellU{0}{0}
		\node at (6,0) {$\dots$};
		\strand{16}{2}
		\strand{20}{2}
		\strand{24}{2}
		\end{scope}
	\end{tikzpicture}.
\end{equation}
Using the algebraic expressions \eqref{eq:Talg}, \eqref{eq:Sialg} and \eqref{eq:Uialg} as well as the properties of the projector $P_k$, we can write
\begin{equation}
	TS_{1}S_{2}\dots S_{i}U_{i} 
	= P_k\si_{k,i+1}P_k\si_{k,i}\si_{k-1,i}\dots \si_{k-i+1,i} P_k.
\end{equation}
Using now the property \eqref{eq:PkPk-1} for the middle projector, we get
\begin{equation}
	TS_{1}S_{2}\dots S_{i}U_{i} 
	= q^{1-k}[k]_q^{-1}\sum_{m=1}^{k} q^{k-m}P_k\si_{k,i+1}(\si_{k-1}\si_{k-2}\dots\si_{m})\si_{k,i}\si_{k-1,i}\dots \si_{k-i+1,i} P_k,
\end{equation} 
where the interior of the parenthesis in the case $m=k$ is understood to be $1$.  
Separating the previous sum in two at $m=k-i$, and absorbing braid generators in the right-most projector $P_k$ for the terms with $m<k-i$, we find
\begin{align}
	TS_{1}S_{2}\dots S_{i}U_{i}=& [k]_q^{-1}[k-i]_qP_k\si_{k,i+1}(\si_{k-1}\si_{k-2}\dots\si_{k-i})\si_{k,i}\si_{k-1,i}\dots \si_{k-i+1,i} P_k \label{eq:TSU}  \\
	+&q[k]_q^{-1}\sum_{m=k-i+1}^{k} q^{-m}P_k\si_{k,i+1}(\si_{k-1}\si_{k-2}\dots\si_{m})\si_{k,i}\si_{k-1,i}\dots \si_{k-i+1,i} P_k. \nonumber   
\end{align}
All terms in the sum of the previous equation do not act on the $k-i$ first fused strands. Therefore, these terms correspond to diagrams where only the $i$ last fused strands go out of the top and bottom ellipses, and they can be obtained from a multiplication of the diagram $U_i$ with diagrams $S_j$. Hence, by the induction hypothesis, they belong to the image of $\phi$. Since we have supposed that $i<k$ and since $q^{2m}-1$ is invertible for $m=1,\dots,k$, the first term in \eqref{eq:TSU} is, up to an invertible factor,
\begin{equation}
	P_k\si_{k,i+1}\si_{k-1}\si_{k,i}\si_{k-2}\si_{k-1,i}\dots \si_{k-i}\si_{k-i+1,i} P_k = P_k\si_{k,i+1}\si_{k-1,i+1}\si_{k-2,i+1}\dots \si_{k-i,i+1} P_k = U_{i+1}.
\end{equation}
Therefore, the element $U_{i+1}$ belongs to the image of $\phi$. By induction, we conclude that all the elements $U_{i}$ for $i=1,2,\dots\min \{k,n\}$ belong to the image.   

\paragraph{Injectivity of $\phi$.}
At this point, we have shown that $\phi:\cA_n^{(k)} \to \fH$ is a surjective morphism. By comparing the cardinality \eqref{eq:cardavoid-k} of the spanning set \eqref{eq:spanAnk} for $\cA_n^{(k)}$ with the dimension \eqref{eq:dimfH} of $\fH$, we deduce that $\dim(\cA_n^{(k)}) \leq \dim(\fH)$. Hence both algebras have the same dimension and $\phi$ is an isomorphism.

\subsection{Towards a definition of $\cA_n^{(k)}$ over $\bC[q^{\pm1}]$}\label{subsec:conj}

The fused Hecke algebra $H_{k,n}$ is not directly defined over $\bC[q^{\pm1}]$. Nevertheless, the presentation by generators and relations of Definition \ref{def-Ank} could be taken as it is over $\bC[q^{\pm1}]$. However, note that the resulting algebra would then not be free over $\bC[q^{\pm1}]$ with the correct dimension, that is, Corollary \ref{coro-basis} would not be true over $\bC[q^{\pm1}]$ since we used that $[k]_q!$ is invertible in $\bC^{(k)}[q^{\pm1}]$ to prove Proposition \ref{prop-spanAnk}.

So we believe that Definition \ref{def-Ank} is not the correct one to take over $\bC[q^{\pm1}]$. The key to this problem is the following conjectural result.
\begin{conj}\label{conj:Ek+1}
In $\cA_{k+1}$, when $(\alpha_1,\alpha_2)$ are specialised to $(q^{-2},q^{2k})$, the element $E_{k+1}^{(q,\alpha_1)}$ factorises as:
\[E_{k+1}^{(q,\alpha_1)}=[k+1]_q! \tilde{E}_{k+1}^{(q,\alpha_1)}\,,\]
where $\tilde{E}_{k+1}^{(q,\alpha_1)}$ is in $\cA_{k+1}$ with coefficients in $\bC[q^{\pm1}]$.
\end{conj}
To support this conjecture, we have checked it explicitly for small values of $k$. Note moreover that it generalises Corollary \ref{coro:Ek+1} which already identified the factor $[k+1]_q$ (but not the full $q$-factorial). Finally, we are able to prove this statement in the quotient $\cC_{n,2}$ of $\cA_n$ (see Section \ref{subsecboundaryseam}).

Just for this subsection, we are going to assume that this conjecture is true, thereby allowing to define, when $\alpha_1,\alpha_2$ are specialised as before, an element $\tilde{E}_{k+1}^{(q,\alpha_1)}$ in $\cA_n$ with coefficients in $\bC[q^{\pm1}]$. Now the correct definition over $\bC[q^{\pm1}]$ of the algebra $\cA_n^{(k)}$ that we promote is the following.
\begin{defi}\label{def-Ankconj}
The algebra $\cA_n^{(k)}$ is the specialisation over $\mathbb{C}[q^{\pm1}]$ of $\cA_n$ corresponding to $(\alpha_1,\alpha_2)=(q^{-2},q^{2k})$ with the additional defining relation if $n>k$:
\begin{equation}
	\tilde{E}_{k+1}^{(q,\alpha_1)}=0\ . \label{eq:relidem2cpnj}
\end{equation}  
\end{defi}
Now, we can prove the analogue of Corollary \ref{coro-basis}.
\begin{prop}
The algebra $\cA_n^{(k)}$ is free over $\bC[q^{\pm1}]$ with basis:
\begin{equation*}
\{g_\omega\ |\ \omega\in B_n(\bar{1}\bar{2},\overline{k+1}\bar{k}\dots\bar{1})\}\ ,
\end{equation*}
\end{prop}
\begin{proof}
First the above set is now a spanning set over $\bC[q^{\pm1}]$. Indeed, following the proof of Proposition \ref{prop-spanAnk}, we see that we have now removed all factors in front of the element in $\cA_{k+1}$ that we need to rewrite using $\tilde{E}_{k+1}^{(q,\alpha_1)}=0$. So the same proof works now over $\bC[q^{\pm1}]$. The freeness follows immediately from the already proved isomorphism with $H_{k,n}$ over the field of fractions $\bC(q)$.
\end{proof}

\section{Centralisers of $U_q(\mathfrak{gl}_N)$ and the boundary seam algebra ($N=2$)}\label{sec-cent}

Let $N>1$ and let $k>0$. In this final section, we combine the preceding sections to complete the following square by defining the algebras $\cC^{(k)}_{n,N}$
\[\begin{array}{ccc}
\cA_n & \rightsquigarrow & \cA^{(k)}_n \\[0.5em]
 \big\downarrow &&  \big\downarrow\\[0.5em]
 \cC_{n,N} & \rightsquigarrow & \cC^{(k)}_{n,N}
 \end{array}\]
and we show their connections with the centralisers of $U_q(gl_N)$ as discussed in the introduction. We then study in details the case $N=2$ to show that we have finally recovered the so-called boundary seam algebra from \cite{MRR}. 

As in the preceding section, we are going to work, unless otherwise specified over the ring $\bC^{(k)}[q^{\pm1}]$.

\subsection{Definition of $\cC^{(k)}_{n,N}$}\label{subsecCnNk}

The following definition has two equivalent forms, due to the two paths in the square above to reach $\cC^{(k)}_{n,N}$. Recall that the specialisation and the relation \eqref{eq:relidem2-2} were by definition how to go from $\cA_n$ to $\cA_n^{(k)}$, while the relations \eqref{eq:relquotient2N-2}-\eqref{eq:relquotient3N-2} were by definition how to go from $\cA_n$ to $\cC_{n,N}$.
\begin{defi}\label{def-CnNk}
Over the ring $\bC^{(k)}[q^{\pm1}]$, \hfill \break
$\bullet$ the algebra $\cC_{n,N}^{(k)}$ is the specialisation of $\cC_{n,N}$ for $\alpha_1=q^{-2}$ and $\alpha_2=q^{2k}$\,, with the additional defining relation if $n>k$:
\begin{equation}
	{E'}_{k+1}^{(q,\alpha_1)}=0\ ; \label{eq:relidem2-2}
\end{equation}  
$\bullet$ equivalently, the algebra $\cC_{n,N}^{(k)}$ is the quotient of $\cA_n^{(k)}$ by the relations:
\begin{align}
&\tilde{E}_N^{(-q^{-1},\alpha_1)}=0\,, \label{eq:relquotient2N-2}\\
&\Lambda^{-q^{-1}}_{N+1}(g_1,\dots,g_N)=0\,. \label{eq:relquotient3N-2} 
\end{align}
\end{defi}

\begin{rem}
Exactly as discussed in Section \ref{subsec:conj} for the algebra $\cA_n^{(k)}$, we emphasize that a good definition of $\cC^{(k)}_{n,N}$ over $\bC[q^{\pm1}]$ would require to prove that the idempotent $E_{k+1}^{(q,\alpha_1)}$ factorises as:
\[E_{k+1}^{(q,\alpha_1)}=[k+1]_q! \tilde{E}_{k+1}^{(q,\alpha_1)}\ .\]
Then we would define $\cC_{n,N}^{(k)}$ by replacing \eqref{eq:relidem2-2} by $\tilde{E}_{k+1}^{(q,\alpha_1)}=0$.
\end{rem}

\subsection{Isomorphism with the centralisers}

We start by relating $\cC_{n,N}^{(k)}$ to the fused Hecke algebra $\fH$.

\begin{prop}\label{prop:CnN}
The algebra $\cC_{n,N}^{(k)}$ is isomorphic to the quotient of $\fH$ by the relations
\begin{align}
P_k\Lambda_{N+1}(\sigma_k,\dots,\sigma_{k+N-1})P_k=0 \, , \label{eq:relASfH1} \\
P_k\Lambda_{N+1}(\sigma_{k+1},\dots,\sigma_{k+N})P_k=0 \, . \label{eq:relASfH2}
\end{align}
\end{prop}   
\proof
From the isomorphism of $\cA_n^{(k)}$ with $\fH$, it remains only to prove that the quasi-idempotent $\tilde{E}_N^{(-q^{-1},\alpha_1)}$ and the antisymmetriser $\Lambda^{-q^{-1}}_{N+1}(g_1,\dots,g_N)$ of $\cA_n^{(k)}$ are mapped to the correct elements in $\fH$. 

First, it is directly seen that 
\begin{equation}
\phi(\Lambda_{N+1}(g_1,\dots,g_N)) = \Lambda_{N+1}(S_1,\dots,S_N) = P_k\Lambda_{N+1}(\sigma_{k+1},\dots,\sigma_{k+N})P_k \, .
\end{equation}
Then, using an explicit basis for $\fS_{N+1}$, it is seen that
\begin{align}
\Lambda_{N+1}(\sigma_k,\dots,\sigma_{k+N-1})=
\Lambda_{N}(\sigma_{k+1},\dots,\sigma_{k+N-1})(1-q^{-1}\si_{k}+\cdots +(-q^{-1})^N\si_{k}\si_{k+1}\dots \si_{k+N-1} ).
\end{align}
Therefore, by definition of the elements $S_i$ and $T$ and by the properties of $P_k$ we can write
\begin{equation} \label{eq:PkASPk}
	P_k\Lambda_{N+1}(\sigma_k,\dots,\sigma_{k+N-1})P_k =  \Lambda_{N}(S_1,\dots,S_{N-1})(P_k-q^{-1}T+\cdots +(-q^{-1})^NTS_{1}\dots S_{N-1} ).
\end{equation}
Define the element $t$ of $\cA_n^{(k)}$ by
\begin{equation} \label{eq:tg0}
	g_0=(q^{-1}\alpha_2-q\alpha_1)t+q^2\alpha_1\ ,
\end{equation} 
where we recall that $\alpha_1=q^{-2}$ and $\alpha_2=q^{2k}$. Since $q^{2k}-1$ is invertible, this indeed defines the element $t$. Then we can rewrite the formula \eqref{eq:idemENtilde1} obtained for $\tilde{E}_N^{(-q^{-1},\alpha_1)}$ as
\begin{equation} \label{eq:ENtilde2}
	\tilde{E}_N^{(-q^{-1},\alpha_1)}=\Lambda_N(g_1,\dots,g_{N-1})\cdot (\alpha_2-q^2\alpha_1)\left(1-q^{-1}\sum_{i=0}^{N-1}(-q^{-1})^{i}tg_1\dots g_i\right)\, .
\end{equation}
Note that, because of \eqref{eq:relST} (taking into account the specialisation of $\al_1$ and $\al_2$), the image of $t$ under the map $\phi$ is $T$. Therefore, by comparing \eqref{eq:PkASPk} with \eqref{eq:ENtilde2}, it is seen that
\begin{equation}
	\phi(\tilde{E}_N^{(-q^{-1},\al_1)})=(q^{2k}-1)  P_k\Lambda_{N+1}(\sigma_k,\dots,\sigma_{k+N-1})P_k\ .
\end{equation}
We conclude using again that $q^{2k}-1$ is invertible.   
\endproof

\paragraph{Example.} Take $N=2$ and let us write in diagrams the first relation \eqref{eq:relASfH1}. This looks like:
\begin{align}
&\begin{tikzpicture}[scale=0.3,baseline={([yshift=\eseq]current bounding box.center)}]
		\ellk{0}{0}
		\ellstrand{0}{0}
		\slab{(4,3)}{$1$}
		\strand{4}{2}
		\slab{(8,3)}{$2$}
		\strand{8}{2}
	\end{tikzpicture}
	-q^{-1}
	\begin{tikzpicture}[scale=0.3,baseline={([yshift=\eseq]current bounding box.center)}]
		\ellk{0}{0}
		\ellocross{0}{0}
		\slab{(4,3)}{$1$}
		\slab{(8,3)}{$2$}
		\strand{8}{2}
	\end{tikzpicture}
	-q^{-1}
	\begin{tikzpicture}[scale=0.3,baseline={([yshift=\eseq]current bounding box.center)}]
		\ellk{0}{0}
		\ellstrand{0}{0}
		\slab{(4,3)}{$1$}
		\slab{(8,3)}{$2$}
		\ocross{4}{2}
	\end{tikzpicture} \nonumber \\
	+q^{-2}
	&\begin{tikzpicture}[scale=0.3,baseline={([yshift=\eseq]current bounding box.center)}] 
	\draw[lightgray,fill=lightgray] (-1,-2) rectangle (1,2);
	\fill (0,2) ellipse (1.4cm and 0.2cm);
	\fill (0,-2) ellipse (1.4cm and 0.2cm);
	\draw[thick] (-1,2) -- (-1,-2);
	\slab{(0,0)}{$k$}	
	\draw[thick] (8,2)..controls +(0,-3) and +(0,+3) .. (1,-2);
	\fill[white] (2.9,-0.3) circle (0.4);
	\draw[thick] (1,2)..controls +(0,-2) and +(0,+2) .. (4,-2);
	\fill (4,2) circle (0.2);
	\fill (4,-2) circle (0.2);
	\slab{(4,3)}{$1$}
	\slab{(8,3)}{$2$}	
	\fill (4,2) circle (0.2);
	\fill (4,-2) circle (0.2);
	\fill (8,2) circle (0.2);
	\fill (8,-2) circle (0.2);
	\fill[white] (5.6,0.2) circle (0.4);
	\draw[thick] (4,2)..controls +(0,-2) and +(0,+2) .. (8,-2);
	\end{tikzpicture}
	+q^{-2}
	\begin{tikzpicture}[scale=0.3,baseline={([yshift=\eseq]current bounding box.center)}] 
	\draw[lightgray,fill=lightgray] (-1,-2) rectangle (1,2);
	\fill (0,2) ellipse (1.4cm and 0.2cm);
	\fill (0,-2) ellipse (1.4cm and 0.2cm);
	\draw[thick] (-1,2) -- (-1,-2);
	\slab{(0,0)}{$k$}	
	\draw[thick] (4,2)..controls +(0,-2) and +(0,+2) .. (1,-2);
	\draw[thick] (8,2)..controls +(0,-2) and +(0,+2) .. (4,-2);
	\fill[white] (2.9,0.25) circle (0.4);
	\fill[white] (5.7,-0.2) circle (0.4);
	\draw[thick] (1,2)..controls +(0,-3) and +(0,+3) .. (8,-2);
	\fill (4,2) circle (0.2);
	\fill (4,-2) circle (0.2);
	\slab{(4,3)}{$1$}
	\slab{(8,3)}{$2$}	
	\fill (4,2) circle (0.2);
	\fill (4,-2) circle (0.2);
	\fill (8,2) circle (0.2);
	\fill (8,-2) circle (0.2);
	\end{tikzpicture}
	-q^{-3}
	\begin{tikzpicture}[scale=0.3,baseline={([yshift=\eseq]current bounding box.center)}] 
	\draw[lightgray,fill=lightgray] (-1,-2) rectangle (1,2);
	\fill (0,2) ellipse (1.4cm and 0.2cm);
	\fill (0,-2) ellipse (1.4cm and 0.2cm);
	\draw[thick] (-1,2) -- (-1,-2);
	\slab{(0,0)}{$k$}
	\slab{(4,3)}{$1$}
	\strand{4}{2}
	\slab{(8,3)}{$2$}
	\draw[thick] (1,-2)..controls +(0,1) and +(-1,0) .. (2.5,-0.6) -- (3.7,-0.6) (4.3,-0.6) --  (5,-0.6) ..controls +(2,0) and +(0,-2) .. (8,2);
	\fill (8,2) circle (0.2);
	\fill (8,-2) circle (0.2);
	\fill[white] (4,0.6) circle (0.3);
	\fill[white] (7.3,0) circle (0.3);
	\draw[thick] (1,2)..controls +(0,-1) and +(-1,0) .. (2.5,0.6) -- (5,0.6) ..controls +(2,0) and +(0,2) .. (8,-2);	
	\fill (4,2) circle (0.2);
\fill (4,-2) circle (0.2);
	\end{tikzpicture} =0.
\end{align}
In words, we plug in the usual $q$-antisymmetriser (here on $3$ strands) using the last strand coming out of the ellipse as a first strand. In terms of the generators $U_1:=T$ and $S_1$ of $H_{k,2}$, the previous relation is:
\begin{equation}
	P_k-q^{-1}T-q^{-1}S_1+q^{-2}S_1T+q^{-2}TS_1 - q^{-3}S_1TS_1 =0. \label{eq:relantisym}
\end{equation}
A similar description works for any $N\geq 2$.

The second relation \eqref{eq:relASfH2} is just the usual $q$-antisymmetriser on $N+1$ strands, which is plugged in using the $N+1$ first dots (and not using at all the strands coming out of the ellipse).

\paragraph{Isomorphism with the centraliser.} In this paragraph only, we will work over the field of fractions $\bC(q)$. Using the notations of the introduction, consider the following centraliser:
\[\mathcal{Z}_{k,n,N}=\text{End}_{U_q(gl_N)}\bigl(L^N_{(k)}\otimes (L^N)^{\otimes n}\bigr)\ .\]
Combining what we have obtained so far with the results from \cite{CP} on these centralisers, we get the following description of $\mathcal{Z}_{k,n,N}$.
\begin{coro}
For all $k,n,N$ as before, we have that $\mathcal{Z}_{k,n,N}$ is isomorphic to $\cC_{n,N}^{(k)}$.
\end{coro}
\begin{proof}
From \cite[\S 5]{CP}, we have that the fused Hecke algebra $H_{k,n}$ surjects onto the centraliser $\mathcal{Z}_{k,n,N}$. Moreover, it is also clear from this construction that both relations \eqref{eq:relASfH1}-\eqref{eq:relASfH2} are satisfied in the image (since the expressions between the projectors are already $0$ in the usual Schur--Weyl duality with $U_q(gl_N)$). Moreover, it was proved in \cite[\S 9]{CP} that, for $q^2$ not a root of unity or over $\bC(q)$, the first relation \eqref{eq:relASfH1} is enough to generate the kernel, and this proves that the quotient of $H_{k,n}$ by \eqref{eq:relASfH1} and \eqref{eq:relASfH2} is isomorphic to $\mathcal{Z}_{k,n,N}$. With Proposition \ref{prop:CnN}, this concludes the proof.
\end{proof}

\begin{rem}
The proof shows that, over $\bC(q)$ or for $q^2$ not a root of unity, the second relation \eqref{eq:relASfH2} is implied by the first. This was already noticed at the level of $\cC_{n,N}$, where it was shown, using the semisimple representation theory in Section \ref{sec:C} that Relation \eqref{eq:relquotient2N-2} implies \eqref{eq:relquotient3N-2}.
\end{rem}

\begin{rem}
The representation theory of $\cC_{n,N}^{(k)}$ over $\bC(q)$ or when $q^2$ is not a root of unity is described as follows. Starting with the algebra $\cA_n^{(k)}$, which is the fused Hecke algebra, for which the irreducible representations were indexed by partitions $\lambda\vdash k+n$ with $\lambda_1\geq k$, we simply remove all those which have strictly more than $N$ lines. This is in agreement with the known decomposition of the tensor product of $U_q(gl_N)$-representations. We will give more details for $N=2$ in Section \ref{subsecboundaryseam} below. 
\end{rem}

\subsection{The boundary seam algebra ($N=2$)}\label{subsecboundaryseam}

For $N=2$, using the methods and the terminology of \cite{CP}, the centraliser $Z_{k,n,2}$ could be called the fused Temperley--Lieb algebra, since it can be described by multiplying the usual Temperley--Lieb algebra by a suitable projector on the left and on the right. In our case here, where only the first representation is fused, the fused Temperley--Lieb algebra was introduced in \cite{MRR}  and called the boundary seam algebra (see also \cite{LS,L-L}). We will show how it is recovered as the algebra $\cC_{n,2}^{(k)}$.

First, recall that the algebra $\cC_{n,2}$ was identified in Section \ref{sec:C} as the one-boundary Temperley--Lieb algebra, using the following change of generators:
\begin{equation}
	e_0:=\alpha_2-g_0, \quad e_i:=q-g_i, \quad i=1,2,\dots,n-1. \label{eq:ei2}
\end{equation}
The presentation of $\cC_{n,2}$ in terms of these generators was given explicitly in \eqref{eq:TLrel1}--\eqref{eq:TLrel5}. Here we complete the presentation of $\cC^{(k)}_{n,2}$ in terms of the same generators.

\begin{prop}\label{prop:boundaryseam}
The algebra $\cC^{(k)}_{n,2}$ is the specialisation of the one-boundary Temperley--Lieb algebra $\cC_{n,2}$ corresponding to $\alpha_1=q^{-2}$ and $\alpha_2=q^{2k}$, and the additional relation, if $n\geq k+1$,
\begin{equation}
	u_1u_2\dots u_{k+1}=0 \ , \label{eq:Bnk2}
\end{equation}
where, for $m=0,\dots,k$,
\begin{equation}
u_{m+1}:=\sum_{r=0}^{m-1}(-q)^{r}\left(1-q^{2(m-r)}\frac{\alpha_1}{\alpha_2}\right)e_{m}e_{m-1} \dots e_{m+1-r}+(-q)^{m}\alpha_2^{-1}e_{m}e_{m-1}\dots e_0 \ .
\end{equation}
\end{prop}
\begin{proof}
According to Definition \ref{def-CnNk}, it remains to describe, if $n\geq k+1$, the following relation of $\cC^{(k)}_{n,2}$ in terms of the generators $e_0,e_1,\dots,e_{n-1}$:
\[\frac{1}{[k+1]_q}E_{k+1}^{(q,\alpha_1)}=0\ .\]
We have done most of the work in Proposition \ref{prop-idempCn2} which gives that, for any $1\leq m<n$, we have in $\cC_{n,2}$
\begin{align} 
&E_{m+1}^{(q,\alpha_1)}=q^{\frac{m(m+1)}{2}}[m+1]_q!\tilde{E}_{m+1}^{(q,\alpha_1)}\,,\label{eq:idemENrenormCn2-b}
\end{align}
where $\tilde{E}_{m+1}^{(q,\alpha_1)}$ is defined recursively by $\tilde{E}_{1}^{(q,\alpha_1)}=(1-\alpha_{2}^{-1}g_0)$ and
\begin{align} 
&\tilde{E}_{m+1}^{(q,\alpha_1)}=\tilde{E}_{m}^{(q,\alpha_1)}\bigl((1-q^2)(1+q g_{m}+\dots +q^{m-1}g_{m}\dots g_2)+q^{m}g_{m}\dots g_1(1-\alpha_{2}^{-1}g_0)\bigr)\ .\label{eq:idemENtildeCn2-b}
\end{align}
We rewrite this recursive definition using \eqref{eq:ei2} together with the properties $\tilde{E}_{m}^{(q,\alpha_1)}e_0=(\alpha_2-\alpha_1)\tilde{E}_{m}^{(q,\alpha_1)}$ and $\tilde{E}_{m}^{(q,\alpha_1)}e_i=0$ for $1\leq i\leq m-1$. As an intermediate step, it is found that, for $1 \leq i \leq m+1$,
\begin{equation}
\tilde{E}_m^{(q,\alpha_1)} q^{m+1-i}(q-e_m)(q-e_{m-1})\dots (q-e_i) = \tilde{E}_m^{(q,\alpha_1)} \sum_{r=0}^{m+1-i} (-1)^{r}q^{2(m+1-i)-r} e_m e_{m-1} \dots e_{m+1-r} \ .  \label{eq:intstep}
\end{equation}
Using \eqref{eq:intstep} in \eqref{eq:idemENtildeCn2-b}, and rearranging sums, one arrives at the result $\tilde{E}_{m+1}^{(q,\alpha_1)} = \tilde{E}_m^{(q,\alpha_1)} u_{m+1}$ with $u_{m+1}$ as in the proposition. Now up to some unnecessary invertible power of $q$, the relation reads:
\[[k]_q!\tilde{E}_{k+1}^{(q,\alpha_1)}=[k]_q!u_1u_2\dots u_{k+1}=0\ .\]
The claim follows from the invertibility of $[k]_q!$ in the ring $\bC^{(k)}[q^{\pm1}]$.
\end{proof}
The elements $u_m$ and $\tilde{E}_{m}^{(q,\alpha_1)}$ appear in \cite{L-L} (up to global factors of $\alpha_2$ and $-q^{-1}$) as generalized Wenzl-Jones factors and generalized Wenzl-Jones projectors respectively for the one-boundary Temperley--Lieb algebra. Now, using the preceding proposition, it can be directly verified that the following mappings give an anti-isomorphism from $\cC_{n,2}^{(k)}$ to the boundary seam algebra (with the notations of \cite{L-L})
\begin{equation}
	e_0 \mapsto q^{k-1}(q-q^{-1})e_n, \qquad e_i \mapsto e_{n-i} \quad 1 \leq i \leq n-1.
\end{equation} 

Let us also mention that a recursive formula, similar to \eqref{eq:Erec2} in $\cH{n}$, holds for $\tilde{E}_{m}^{(q,\alpha_1)}$ when $q$ is not a root of unity (or over $\bC(q)$)
in a specialisation such that $\prod_{r=0}^{m-2}\left(1-q^{2r}\frac{\alpha_1}{\alpha_2}\right) \neq 0$:
\begin{equation}
	\tilde{E}_{m}^{(q,\alpha_1)} = \left(1-q^{2(m-1)}\frac{\alpha_1}{\alpha_2}\right)\tilde{E}_{m-1}^{(q,\alpha_1)} - q\frac{\tilde{E}_{m-1}^{(q,\alpha_1)}e_{m-1}\tilde{E}_{m-1}^{(q,\alpha_1)}}{\prod_{r=0}^{m-3}\left(1-q^{2r}\frac{\alpha_1}{\alpha_2}\right)} \ . 
\end{equation} 

\paragraph{Semisimple representation theory.} In this paragraph, we work over $\bC(q)$ or we assume that $q^2$ is not a root of unity. The representation theory of the boundary seam algebra $\cC_{n,2}^{(k)}$ is easily obtained from the one of the fused Hecke algebra $\cA_n^{(k)}$. Recall from Section \ref{sec:fusedH} that the irreducible representations of $\cA_n^{(k)}$ were indexed by partitions $\lambda\vdash k+n$ with $\lambda_1\geq k$. The quotient giving the boundary seam algebra $\cC_{n,2}^{(k)}$ consists simply in removing all those which have strictly more than two lines.

The beginning of the Bratteli diagram, for example for $k=3$, of the algebras $\cC_{n,2}^{(k)}$ is given below
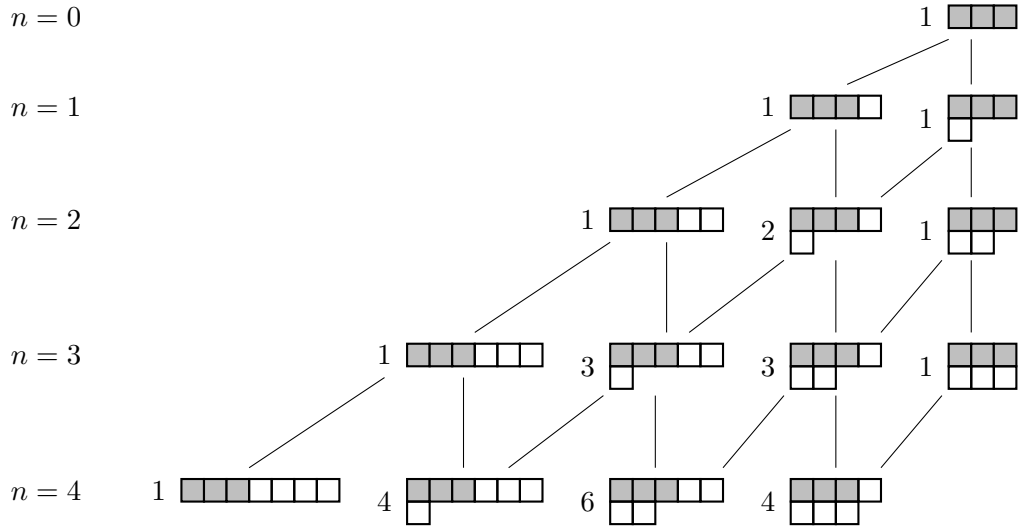
\begin{center}
 \begin{tikzpicture}[scale=0.3]
\mdiagp{2}{0}{3};\node at (1,-0.5) {$1$};
\draw (2,-1.5) -- (-2.5,-3.5);\draw (3,-1.5) -- (3,-3.5);
\mdiag{-5}{-4}{4};\node at (-6,-4.5) {$1$};\mdiaggp{2}{-4}{3}{1};\node at
(1,-5) {$1$};

\draw (-5,-5.5) -- (-10.5,-8.5);\draw (-3,-5.5) -- (-3,-8.5);\draw
(1.7,-6.3) -- (-1,-8.5);\draw (3,-6.3) -- (3,-8.5);

\mdiag{-13}{-9}{5};\node at (-14,-9.5) {$1$};\mdiagg{-5}{-9}{4}{1};\node
at (-6,-10) {$2$};\mdiaggp{2}{-9}{3}{2};\node at (1,-10)
{$1$};

\draw (-13,-10.5) -- (-19,-14.5);\draw (-10.5,-10.5) --
(-10.5,-14.5);\draw (-5.3,-11.3) -- (-9.5,-14.5);\draw (-3,-11.3) --
(-3,-14.5);
\draw (1.7,-11.3) -- (-1,-14.5);\draw (3,-11.3) -- (3,-14.5);

\node at (-23,-15.5) {$1$};\mdiag{-22}{-15}{6};
\node at (-14,-16) {$3$};\mdiagg{-13}{-15}{5}{1};

\node at (-6,-16) {$3$};\mdiagg{-5}{-15}{4}{2};

\node at (1,-16) {$1$};\mdiaggp{2}{-15}{3}{3};

\draw (-23,-16.5) -- (-29,-20.5);\draw (-19.5,-16.5) --
(-19.5,-20.5);\draw (-13.3,-17.3) -- (-17.5,-20.5);\draw (-11,-17.3) --
(-11,-20.5);
\draw (-5.3,-17.3) -- (-8,-20.5);\draw (-3,-17.3) -- (-3,-20.5);
\draw (1.7,-17.3) -- (-1,-20.5);

\node at (-33,-21.5) {$1$};\mdiag{-32}{-21}{7};
\node at (-23,-22) {$4$};\mdiagg{-22}{-21}{6}{1};

\node at (-14,-22) {$6$};\mdiagg{-13}{-21}{5}{2};

\node at (-6,-22) {$4$};\mdiagg{-5}{-21}{4}{3};

\node at (-38,-0.5) {$n=0$};\node at (-38,-4.5) {$n=1$};\node at
(-38,-9.5) {$n=2$};\node at (-38,-15.5) {$n=3$};
\node at (-38,-21.5) {$n=4$};
\end{tikzpicture}
\end{center}
We can see the special behaviour starting at the value $n=k+1=4$. The irreducible representations, at level $n$, are indexed by a positive integer $h$, which is the number of boxes added in the second row, satisfying $0\leq 2h\leq k+n$. It is easy to see recursively from the branching graph that the dimension of the corresponding irreducible representation is:
\[\binom{n}{h}-\binom{n}{h-k-1}\ ,\]
with the understanding that $\binom{n}{h-k-1}=0$ if $h\leq k$. Note that comparing with \cite{MRR}, our $h$ is their $(n+k-d)/2$. The dimension of the algebra is:
\begin{equation}\label{dimCn2}
\dim C_{n,2}^{(k)}=\binom{2n}{n}-\binom{2n}{n-k-1}\ .
\end{equation}

\paragraph{Definition over $\bC[q^{\pm1}]$.}
We have originally defined $\cC_{n,2}^{(k)}$ over $\bC^{(k)}[q^{\pm1}]$. The presentation put forward in Proposition \ref{prop:boundaryseam} can be used without problem over $\bC[q^{\pm1}]$. In our notation, this means to define the algebra $\cC_{n,2}^{(k)}$ over $\bC[q^{\pm1}]$ as follows.
\begin{defi}\label{def-CnNk-Cq}
Over $\bC[q^{\pm1}]$, the algebra $\cC_{n,2}^{(k)}$ is the specialisation of $\cC_{n,2}$ for $\alpha_1=q^{-2}$ and $\alpha_2=q^{2k}$\,, with the additional defining relation if $n>k$:
\begin{equation}
	\tilde{E}_{k+1}^{(q,\alpha_1)}=0\ , \label{eq:relidem2-22}
\end{equation}  
where the renormalised quasi-idempotent $\tilde{E}^{(q,\alpha_1)}_{k+1}$ was obtained in Proposition \ref{prop-idempCn2} and recalled in \eqref{eq:idemENtildeCn2-b}. 
\end{defi}
With this definition, we can prove that we get an algebra which is free over $\bC[q^{\pm1}]$ with dimension equal to \eqref{dimCn2}. In fact, we may check that the following set of elements is a $\bC[q^{\pm1}]$-basis:
\[
[n_1,m_1][n_2,m_2]\dots [n_r,m_r]\ \ \ \ \ \ \text{with}\ \left\{\begin{array}{l} 0\leq n_1<n_2<\dots<n_r\leq n-1\ \ \text{and}\ \ m_i\leq n_i\ ,\\[0.4em]
0=m_1=\dots=m_s<m_{s+1}<\dots<m_r\ \ \ s<k+1\ .
\end{array}\right.\]
Without the condition $s<k+1$ in the second line, we already know that this set is a spanning set for $\cC_{n,2}$, see \eqref{basis-cC2}. The relation $\tilde{E}_{k+1}^{(q,\alpha_1)}=0$ further allows to rewrite any element $[n_1,0][n_2,0]\dots [n_{k+1},0]$ in terms of elements with fewer $g_0$ (smaller $s$). This works over $\bC[q^{\pm1}]$ since the element we need to rewrite appears with an invertible coefficient in $\tilde{E}_{k+1}^{(q,\alpha_1)}$. We refer to the proof of Proposition \ref{prop-spanAnk} for more details. The above set is of the correct cardinality \cite[Appendix B]{MRR}, that is \eqref{dimCn2}, and thus is a basis over $\bC[q^{\pm1}]$.

\begin{rem}
If we specialise $q$ to a complex number such that $q^{2i}\neq 1$ for $i=1,\dots,k$, of course Definition \ref{def-CnNk-Cq} recovers Definition \ref{def-CnNk}. But now Definition \ref{def-CnNk-Cq} allows to consider the cases where $q^{2i}=1$ for some $i=1,\dots,k$. We note that we differ here from \cite{MRR} where the defining relations, when $q^{2i}\neq 1$, were modified according to the value of $q$ and the dimension of the algebra resultingly depended on $q$.
\end{rem}

%
%
%
%
%
%


\begin{thebibliography}{99}

\bibitem{ALZ} 
H.H. Andersen, G.I. Lehrer and R.B. Zhang, 
\emph{Cellularity of certain quantum endomorphism algebras}, 
Pacific J. Math. 279(1-2) (2015), 11--35, 
\href{https://arxiv.org/abs/1303.0984}{\texttt{arXiv:1303.0984}}.

\bibitem{Ari} 
S. Ariki, 
\emph{On the semi-simplicity of the Hecke algebra of $(\mathbb{Z}/r\mathbb{Z})\wr S_n$}, 
J. Algebra 169 (1994), 216--225.

\bibitem{AK} 
S. Ariki and K. Koike, 
\emph{A Hecke algebra of $(\mathbb{Z}/r\mathbb{Z}) \wr \mathfrak{S}_n$ and construction of its irreducible representations},
Adv. in Math. 106 (1994), 216--243.

\bibitem{CGS} 
D. Chernyak, A.M. Gainutdinov and H. Saleur, 
\emph{$U_q\mathfrak{sl}_2$-invariant non-compact boundary conditions for the XXZ spin chain}, 
J. High Energ. Phys. 2022, no. 11, Paper No. 16, 64 pp,
\href{https://arxiv.org/abs/2207.12772}{\texttt{arXiv:2207.12772}}.

\bibitem{ChP} 
M. Chlouveraki and G. Pouchin, \emph{Determination of the representations and a basis for the Yokonuma--Temperley--Lieb algebra}, Algebr. Represent. Theory 18 (2015), no. 2, 421--447,
\href{https://arxiv.org/abs/1311.5626}{\texttt{arXiv:1311.5626}}.

\bibitem{CP} 
N. Cramp\'e and L. Poulain d'Andecy, 
\emph{Fused braids and centralisers of tensor representations of $U_q(gl_N)$},
Algebr. Represent. Theor. (2022), 
\href{https://arxiv.org/abs/2001.11372}{\texttt{arXiv:2001.11372}}.

\bibitem{CP2} 
N. Cramp\'e and L. Poulain d'Andecy, 
\emph{Baxterisation of the fused Hecke algebra and $R$-matrices with $gl(N)$-symmetry},
Lett. Math. Phys. 111 (2021), no. 4, Paper No. 92, 21 pp,
\href{https://arxiv.org/abs/2004.05035}{\texttt{arXiv:2004.05035}}.

%

\bibitem{tD} 
T. tom Dieck, 
\emph{Symmetrische Brücken und Knotentheorie zu den Dynkin-Diagrammen vom Typ B}, 
J. reine angew. Math. 451 (1994), 71--88.

\bibitem{Eg} 
E. Egge, 
\emph{Enumerating rc-invariant permutations with no long decreasing subsequences}, 
Ann. Comb. 14 (2010), no. 1, 85--101.

\bibitem{FP} 
S. M. Flores and E. Peltola, 
\emph{Generators, projectors, and the Jones--Wenzl algebra}, 
arXiv preprint \href{https://arxiv.org/abs/1811.12364}{\texttt{arXiv:1811.12364}} (2018).

\bibitem{GP} 
M. Geck and G. Pfeiffer, 
\emph{Characters of finite Coxeter groups and Iwahori-Hecke algebras}. 
London Mathematical Society Monographs. New Series, 21. The Clarendon Press, Oxford University Press, New York, 2000. xvi+446 pp.

\bibitem{Ho} 
P. Hoefsmit, 
\emph{Representations of Hecke algebras of finite groups with BN-pairs of classical type}, 
Ph. D. thesis, University of British Columbia (1974). 

\bibitem{ILZ} 
K. Iohara, G. Lehrer and R. Zhang, \emph{Schur--Weyl duality for certain infinite dimensional $U_q(sl_2)$-modules}, 
arXiv preprint \href{https://arxiv.org/abs/1811.01325}{\texttt{arXiv:1811.01325}} (2018).

\bibitem{Jimbo} 
M. Jimbo, 
\emph{A q-Analogue of $U(\mathfrak{gl}(N+1))$, Hecke algebra, and the Yang--Baxter equation}, 
Lett. Math. Phys. 11 (1986), 247--252.

\bibitem{LV} 
A. Lacabanne and P. Vaz, 
\emph{Schur--Weyl duality, Verma modules, and row quotients of Ariki--Koike algebras}, 
Pacific J. Math. 311 (2021), no. 1, 113--133,
\href{https://arxiv.org/abs/2004.01065}{\texttt{arXiv:2004.01065}}.

\bibitem{LS} 
A. Langlois-Rémillard and Y. Saint-Aubin,
\emph{The representation theory of seam algebras},
SciPost. Phys. 8 (2020),
\href{https://arxiv.org/abs/1909.03499}{\texttt{arXiv:1909.03499}}.

\bibitem{L-L} 
A. Leroux-Lapierre,
\emph{La famille exceptionnelle des algèbres à couture},
Master thesis, Université de Montréal, 2020, 115 p.

\bibitem{MS} 
P.P. Martin and H. Saleur, 
\emph{The blob algebra and the periodic Temperley--Lieb algebra}, 
Lett. Math. Phys. 30 (1994), 189--206,
\href{https://arxiv.org/abs/hep-th/9302094}{\texttt{arXiv:hep-th/9302094}}.

\bibitem{MRR} 
A. Morin-Duchesne, J. Rasmussen and D. Ridout,
\emph{Boundary algebras and Kac modules for logarithmic minimal models},
Nucl. Phys. B 899 (2015), 677--769,
\href{https://arxiv.org/abs/1503.07584}{\texttt{arXiv:1503.07584}}.

\bibitem{NRdG} 
A. Nichols, V. Rittenberg and J. de Gier, 
\emph{One-boundary Temperley--Lieb algebras in the XXZ and loop models}, 
J. Stat. Mech. 0503 (2005) P03003, 
\href{https://arxiv.org/abs/cond-mat/0411512}{\texttt{arXiv:cond-mat/0411512}}.

\bibitem{OEIS} 
OEIS Foundation Inc., 
``The On-Line Encyclopedia of Integer Sequences'', \href{http://oeis.org}{\texttt{http://oeis.org}}.

\bibitem{OP11} 
O.V. Ogievetsky and L. Poulain d'Andecy,
\emph{On representations of cyclotomic Hecke algebras},
Mod. Phys. Lett. A 26 (2011), 795--803,
\href{https://arxiv.org/abs/1012.5844}{\texttt{arXiv:1012.5844}}.

\bibitem{OR} 
R. Orellana and A. Ram, 
\emph{Affine braids, Markov traces and the category $\mathcal{O}$}, 
Algebraic groups and homogeneous spaces, 423--473, Tata Inst. Fund. Res. Stud. Math., 19, Tata Inst. Fund. Res., Mumbai, 2007,
\href{https://arxiv.org/abs/math/0401317}{\texttt{arXiv:math/0401317}}.

\bibitem{P1} 
L. Poulain d'Andecy, 
\emph{Fusion for the Yang--Baxter equation and the braid group}, 
Winter Braids Lect. Notes 7 (2020), Winter Braids X (Pisa, 2020), Exp. No. 3, 49 pp,
\href{https://arxiv.org/abs/2204.03483}{\texttt{arXiv:2204.03483}}.

\bibitem{Ram} 
A. Ram, 
\emph{Skew shape representations are irreducible}, 
Combinatorial and geometric representation theory (Seoul, 2001), 161--189, Contemp. Math., 325, Amer. Math. Soc., Providence, RI, 2003,
\href{https://arxiv.org/abs/math/0401326}{\texttt{arXiv:math/0401326}}.

\bibitem{Res} 
N. Yu. Reshetikhin, 
\emph{Quantized universal enveloping algebras, the Yang--Baxter equations
and invariants of links, I and II}, 
LOMI preprints E-4-87 and E-17-87, Leningrad (1987).

\bibitem{Sim} 
R. Simion,
\emph{Combinatorial statistics on type-B analogues of noncrossing partitions and restricted permutations},
Electronic J. of Comb. 7 (2000).

\bibitem{Spe} 
R.A. Spencer, 
\emph{Modular valenced Temperley--Lieb algebras}, 
arXiv preprint \href{https://arxiv.org/abs/2108.10011}{\texttt{arXiv:2108.10011}} (2021).

\bibitem{Stem} 
J.R. Stembridge,
\emph{Some combinatorial aspects of reduced words in finite Coxeter groups},
Trans. Amer. Math. Soc. 349 (1997) 1285--1332.

\end{thebibliography}
\end{document}